\documentclass[11pt,reqno]{amsart}

\usepackage{amssymb,amsmath,mathrsfs,amsthm}
\usepackage{wrapfig}
\usepackage{graphicx}
\usepackage{bbm}
\usepackage{fullpage}

\usepackage[dvipsnames]{xcolor}

\usepackage{etoolbox}
\patchcmd{\section}{\scshape}{\bfseries}{}{}
\makeatletter
\renewcommand{\@secnumfont}{\bfseries}
\makeatother

\DeclareRobustCommand{\SkipTocEntry}[5]{}

\usepackage[scaled=0.90]{helvet} 

\newtheorem{introtheorem}{Theorem}

\theoremstyle{definition}

\newtheorem*{introdefinition*}{Definition}
\theoremstyle{plain}

\def\F{\mathbf{F}}

\def\text#1{\hbox{#1}}
\def\Q{{\mathbf Q}}
\def\Z{{\mathbf Z}}

\def\N{{\mathbf N}}

\def\Hom{{\rm Hom}}

\def\Ass{{\rm Ass}}

\def\ann{{\rm ann}}

\def\L{{\rm L}}

\theoremstyle{plain}
\newtheorem{theorem}{Theorem}[section]
\newtheorem{proposition}[theorem]{Proposition}
\newtheorem{lemma}[theorem]{Lemma}
\newtheorem{corollary}[theorem]{Corollary}
\theoremstyle{definition}
\newtheorem{definition}[theorem]{Definition}
\newtheorem{example}[theorem]{Example}
\newtheorem{remark}[theorem]{Remark}
\newtheorem{question}[theorem]{Question}

\usepackage{calrsfs}
\usepackage[T1]{fontenc} 
\usepackage{amsfonts}

\renewcommand{\geq}{\geqslant}
\renewcommand{\leq}{\leqslant}

\begin{document}

 \title{Rado's theorem for rings and modules}

\author[J.~Byszewski]{Jakub Byszewski}
\address[J.~Byszewski]{\normalfont Faculty of Mathematics and Computer Science, Institute of Mathematics, Jagiellonian University,
{\L}ojasiewicza~6, 30-348 Krak\'ow, Poland}
\email{jakub.byszewski@uj.edu.pl}

\author[E.~Krawczyk]{El\.zbieta Krawczyk}
\address[E.~Krawczyk]{\normalfont Faculty of Mathematics and Computer Science, Institute of Mathematics, Jagiellonian University,
{\L}ojasiewicza~6, 30-348 Krak\'ow, Poland}
\email{ela.krawczyk@student.uj.edu.pl}

\thanks{Declarations of interest: none}

\maketitle
\begin{abstract} We extend classical results of Rado on partition regularity of systems of linear equations with integer coefficients to the case when the coefficient ring is either an arbitrary integral domain or a noetherian ring. In particular, we show that a system of homogeneous linear equations  over an infinite integral domain is partition regular if and only if the corresponding matrix satisfies the columns conditions. The crucial idea is to study partition regularity for general modules rather than only for rings. Contrary to previous techniques, our approach is independent of the characteristic of the coefficient ring.

\end{abstract}

\section*{Introduction}

We say that a system of equations with variables taking values in a set $S$ is partition regular over $S$ if for every finite partition of  $S$ one cell of the partition contains a solution to the system. Many famous results in Ramsey theory (including Schur's theorem and van der Waerden's theorem) can be stated as saying that a certain system of equations is partition regular. The problem of whether a given system of equations is partition regular or not has been widely studied (see, e.g.\ \cite{BHL2013, BDHL, Deuber1973, Deuber1975, DeuberHindman1987, Moreira2017, Rado1933, Rado1943}). The first general result which concerns partition regularity of  systems of linear equations with integer coefficients over the set of positive integers is due to Rado \cite{Rado1933, Rado1943}. For a single equation, it says that an equation $$a_1 x_1 + a_2 x_2 +\dots +a_n x_n=0$$ with nonzero integer coefficients is partition regular if and only if $\sum_{i\in I} a_i =0$ for some nonempty $I \subset \{1,\dots,n\}$. In general, Rado's theorem states that a system of linear equations of the form $\mathbf{A} \mathbf{x} = \mathbf{0},$ where $\mathbf{A}$ is a matrix with integer entries, is partition regular if and only if the matrix $\mathbf{A}$ satisfies the so-called columns condition, stated below for an arbitrary integral domain.

\begin{definition}\label{def:cc} Let $\mathbf{A}$ be a $k\times l$ matrix with entries in an integral domain $R$. Let $K$ be the fraction field of $R$. Denote the columns of $\mathbf{A}$ by $ \mathbf{c}_1,\dots, \mathbf{c}_l \in R^k$. We say that $\mathbf{A}$ satisfies the \emph{columns condition} if there exists an integer $m\geq 0$ and a partition of the set of column indices $\{1,\dots, l\}=I_0 \cup I_1 \cup \dots \cup I_m$ such that $\sum_{i \in I_0} \mathbf{c}_i =\mathbf{0}$ and such that for $t\in\{1,\dots,m\}$ the vector $\sum_{i \in I_t} \mathbf{c}_i$ lies in the $K$-vector space generated by the columns $c_j$ with $j\in I_0\cup\dots\cup I_{t-1}$. \end{definition}

In this paper we study partition regularity in the more general context of (commutative) rings and modules. Let $R$ be a ring, let $\mathbf{A}$ be a matrix with entries in $R$, and let $M$ be an $R$-module. It is meaningful to ask whether the system of equations $\mathbf{A}\mathbf{m}=\mathbf{0}$ is \emph{partition regular} over $M$ in the sense that for every finite partition of $M$ one cell of the partition contains a nontrivial solution to the system $\mathbf{A}\mathbf{m}=\mathbf{0}$. Such problems have been considered before in specific cases, e.g.\ in  \cite{Rado1943, BDHL}  (for commutative rings), in  \cite{Deuber1973, Deuber1975} (for integral equations over abelian groups) and in \cite{BDH1992} (for vector spaces over finite fields).  Studying partition regularity for modules rather than just for rings  gives us extra technical flexibility and allows us to use notions and methods from commutative algebra. In particular, we solve completely the problem of whether the system $\mathbf{A}\mathbf{x}=\mathbf{0}$ is partition regular over a ring $R$ if $R$ is either an integral domain or a noetherian ring.

Our first main result shows that the columns condition characterises partition regularity of homogeneous linear equations over every infinite integral domain.

\begin{introtheorem}\label{mainthmB} Let $R$ be an infinite integral domain and let
$\mathbf{A}$ be a matrix with entries in $R$. Then the system $\mathbf{A}\mathbf{x}=\mathbf{0}$ is partition regular over $R$ if and only if $\mathbf{A}$ satisfies the columns condition.
\end{introtheorem}

The fact that the system $\mathbf{A}\mathbf{x}=\mathbf{0}$ is partition regular over an infinite integral domain $R$ if $\mathbf{A}$ satisfies the columns condition follows from \cite{BDHL}. In the case when $R$ is a subring of the complex numbers Theorem \ref{mainthmB} has been proved by Rado  \cite{Rado1943}, and the case when $R$ is of characteristic zero can be obtained from it by a compactness argument (using, e.g.\ Lemma \ref{auxlemma:domains}). Thus the main interest of Theorem \ref{mainthmB} is when the ring $R$ has positive characteristic. The advantage of our method is that it provides a uniform approach which works regardless of the characteristic. 

Let us give some details about the proof of Theorem \ref{mainthmB}. The crucial argument is to show that a matrix that does not satisfy the columns condition is not partition regular. To this end, Rado introduced the notion  of $c_p$ colourings, defined as follows: For a prime number $p$, the colouring $c_p$ assigns to a positive integer $n$ the least nonzero digit of $n$ in base $p$. Rado proved that if a system of equations $\mathbf{A}\mathbf{x}=\mathbf{0}$ with integer coefficients is partition regular over the set of positive integers with respect to all the colourings  $c_p$, then the matrix $\mathbf{A}$ satisfies the columns condition. In order to prove Theorem \ref{mainthmB}, we first reduce the problem to the case when $R$ is a finitely generated $\Z$-algebra. We then construct a family of colourings $c_{\mathfrak m}$, where the role of a prime $p$ in Rado's argument is played by an arbitrary maximal ideal $\mathfrak m$ of $R$ such that the local ring $R_{\mathfrak m}$ is regular. We  prove that if a system of equations $\mathbf{A}\mathbf{x}=\mathbf{0}$ is partition regular over $R$ with respect to all the  colourings $c_{\mathfrak m}$, then the matrix $\mathbf{A}$ satisfies the columns condition. 

In order to state our second main result, we recall the notion of an associated prime.  A prime ideal $\mathfrak{p}$ of a ring $R$ is an \emph{associated prime} of an $R$-module $M$ if there exists an element $m\in M$ with $\mathrm{ann}(m)=\mathfrak{p}$, where $\mathrm{ann}(m)=\{r\in R\mid rm=0\}$. A finitely generated module over a noetherian ring has only finitely many associated primes. 
The following result reduces the study of partition regularity for finitely generated modules over noetherian rings to the study of partition regularity over noetherian integral domains.

\begin{introtheorem}\label{mainthmA} Let $M$ be a finitely generated module over a noetherian ring $R$ and let
$\mathbf{A}$ be a matrix with entries in $R$. Then the system $\mathbf{A}\mathbf{m}=\mathbf{0}$ is partition regular over $M$  if and only if there exists an associated prime $\mathfrak{p}$ of $M$  such that the system $\mathbf{A}\mathbf{x}=\mathbf{0}$  is partition regular over $R/\mathfrak{p}$.
\end{introtheorem}

Our study is inspired by a paper of Bergelson, Deuber, Hindman, and Lefmann \cite{BDHL}, where the authors studied equations with coefficients in arbitrary rings. To this end, they introduced the following generalisation of the columns condition. This property is called the columns condition in \cite{BDHL}, but in order to distinguish it from the simpler condition considered in  Definition \ref{def:cc}, we will refer to it as the generalised columns condition. The generalised columns condition is easily seen to be equivalent to the columns condition when $R$ is an infinite integral domain.

\begin{definition}\label{def:gcc} Let $\mathbf{A}$ be a $k\times l$ matrix with entries in a ring $R$. Denote the columns of $\mathbf{A}$ by $ \mathbf{c}_1,\dots, \mathbf{c}_l \in R^k$. We say that $\mathbf{A}$ satisfies the \emph{generalised columns condition} if there exists an integer $m\geq 0$, a partition of the set of column indices $\{1,\dots, l\}=I_0 \cup I_1 \cup \dots \cup I_m$, and elements $d_0,d_1,\dots,d_{m} \in R\setminus \{0\}$ such that the following conditions hold: \begin{enumerate} \item $d_0\cdot \sum_{i \in I_0} \mathbf{c}_i =\mathbf{0}$. \item For $t\in\{1,\dots,m\}$ the vector $d_t \cdot \sum_{i \in I_t} \mathbf{c}_i$ lies in the $R$-module generated by the columns $c_j$ with $j\in I_0\cup\dots\cup I_{t-1}$. \item \label{jestjuzbardzopozno} If $m>0$, then for each $n\geq 0$ the ideal $ d_0 (d_1\cdots d_{m})^n R$ is infinite.\end{enumerate} \end{definition}

In \cite[Theorem 2.4]{BDHL} it is shown that if a matrix $\mathbf{A}$ over an arbitrary ring satisfies the generalised columns condition, then the system of equations $\mathbf{A}\mathbf{x}=\mathbf{0}$ is partition regular over $R$. However, this condition is in general \emph{not necessary}, e.g.\ when $R=\prod_{i=1}^{\infty} {\Z}/4{\Z}$. A ring $R$ is called  a \emph{Rado ring} if the generalised columns condition is equivalent to partition regularity for all matrices with entries in $R$. It follows from Theorem \ref{mainthmB} that every integral domain is a Rado ring. Non-examples of Rado rings have been scarce. In fact, the only previously known example of a non-Rado ring comes from \cite[Theorem 3.5]{BDHL}, where it is shown that the ring $R=\prod_{i=1}^{\infty} {\Z}/n{\Z}$ is a Rado ring if and only if $n$ is squarefree. This example is a bit unsatisfactory since the ring in question is not noetherian. In Theorem \ref{thm:Radonoeth} we classify all noetherian Rado rings and in particular prove that all reduced (i.e.\ without nonzero nilpotents) noetherian rings are Rado. We also show that the ring $R=(\Z/p^2\Z)[X]$ is not a Rado ring.

One of the crucial ideas of this paper is to study partition regularity for general modules rather than just for rings. This allows for a uniform treatment of rings, their ideals, quotient rings and localisations, and often gives us freedom to change the coloured space to some object with better algebraic properties (while keeping the matrix in the same coefficient ring). For instance, in Proposition \ref{prop:PR_of_quotients} we show that for any ring $R$ and any ideal $I\subset R$, a matrix $\mathbf{A}$ (with entries in $R$) which  is partition regular over $R$ has to be partition regular over one of the $R$-modules $I$ or $R/I$. This is the backbone of many inductive arguments that we employ to prove partition regularity results for rings; such arguments appear e.g.\ in the proof of Theorem \ref{mainthmA}  or in the proof of Proposition \ref{thm:characterisation_of_products} (which characterises partition regularity over the infinite product ring $R=\prod_{i=1}^{\infty} {\Z}/n{\Z}$). In the case when the ring $R$ is an integral domain, partition regularity over $R$ is equivalent to partition regularity over the fraction field of $R$ (see Proposition \ref{prop:PR_under_localisation}). This is exploited e.g.\ in  Lemma \ref{auxlemma:domains}, which plays an important role in the proof of Theorem \ref{mainthmB}, where it allows us  to reduce the problem to the case of finitely generated $\mathbf{Z}$-algebras.

We also study partition regularity of nonhomogeneous equations over arbitrary modules. In this case a system of equations $\mathbf{Am}=\mathbf{b}$ with $ \mathbf{0}\neq \mathbf{b}\in M^k$ is called partition regular over $M$ if for any finite colouring of $M$ one cell of the partition contains a solution $\mathbf{m}$. One way for a nonhomogeneous equation to be partition regular is to admit a constant solution $\mathbf m  = (m,\dots,m)$ with all the coordinates equal. In  \cite{Rado1933}, Rado showed that this is the only possibility if $R=M=\Z$. In Theorem \ref{lem:nonhomogenous-one_eq} we rather easily extend this result to the case of an arbitrary module $M$ and a single equation $a_1 m_1 + \dots + a_l m_l = b$ using a colouring result of Straus \cite{Straus}. For systems of equations, we can only prove such a result under certain quite weak assumptions. We state below a slightly simplified version of the result.

\begin{introtheorem}\label{mainthmC} Let $R$ be a ring and let $M$ be an $R$-module.  Let $\mathbf{A}$ be a $k\times l$ matrix with entries in $R$ and let $\mathbf{b}\in M^k$ be nonzero. Assume that one of the following assumptions holds: \begin{enumerate} \item[(a)] either $k=1$; or \item[(b)] $R$ is an integral domain and $M$ is a torsion-free module; or \item[(c)] $R$ is a Dedekind integral domain; or \item[(d)] $R$ is a reduced noetherian ring and $M=R$.\end{enumerate}
Then the system $\mathbf{Am}=\mathbf{b}$ is partition regular over $M$ if and only if it  has a constant solution in $M$.
\end{introtheorem}

We do not know if the assumptions of Theorem \ref{mainthmC} are necessary. In fact, we do not know any examples of modules over which nonhomogeneous equations would be partition regular without admitting constant solutions.\footnote{It has now been shown in \cite{LR-2020} that such modules do not exist (added in proof).}

It might be argued that the definition of partition regularity for nonhomogeneous equations is rather artificial, and that we should insist that the monochromatic solution to the equation  $\mathbf{Am}=\mathbf{b}$ be nonconstant. Note, however, that if $\mathbf{Am}=\mathbf{b}$ admits a constant solution, the set of solutions of $\mathbf{Am}=\mathbf{b}$ is simply a translate of the set of solutions of the homogeneous equation $\mathbf{Am}=\mathbf{0}$. Thus the question of existence of a nonconstant monochromatic solution of $\mathbf{Am}=\mathbf{b}$ (or even a solution with all the variables different) is reduced to the corresponding problem for homogeneous equations.  While we do not study these questions in this paper, we refer the interested reader to \cite{HL06} for the case when $M=\Z$ or $M=\Q$.

We briefly discuss the contents of the paper. 

In Section \ref{sec:bn} we introduce some basic properties of partition regularity over modules. In particular, we show that partition regularity behaves well with respect to short exact sequences of modules, which allows us to perform d\'evissage arguments. Several properties here generalise those proved  by Deuber for abelian groups \cite{Deuber1975}.

 In Section \ref{sec:mod}, we apply these methods to finitely generated modules over noetherian rings and prove Theorem \ref{mainthmA}. 
 
 In Section \ref{sec:dom}, we introduce  $\mathfrak m$-colourings (defined on fields that are finitely generated over $\F_p$ or $\Q$) and use them to prove Theorem \ref{mainthmB}.

 The aim of Section \ref{sec:Rado} is twofold. We first classify noetherian Rado rings, and then characterise partition regularity over the infinite product ring $\prod_{i\in I} \Z/n\Z$. Using module-theoretic techniques we are able to generalise some results and answer some questions from \cite{BDHL}.
 
 Finally, in Section \ref{sec:nonhom} we study nonhomogeneous equations. We use here a classical method to deduce the existence of a constant solution of a system of equations from the existence of such a solution for linear combinations of individual equations. For general modules this method does not always work, and we introduce certain modules that measure obstruction to its applicability. We then show that this obstruction vanishes in the cases considered in Theorem \ref{mainthmC}.

We hope that the paper will also be of interest to readers with little or no background in commutative algebra. For this reason, we have tried to recall all the  notions  and to include precise references for the results that we need. Our general reference in commutative algebra is the book of Eisenbud \cite{book:Eisenbud}.

\subsection*{Notations} All rings are assumed to be commutative and with a  unit. By $\N=\{0,1,\dots\}$ we denote the set of natural numbers, and by $\F_p$ the finite field with $p$ elements. We denote by $R^*$ the group of invertible elements of a ring $R$. Given a quotient map $R\to R/I$, we denote the image of $x\in R$ in  $R/I$ by $\bar{x}$ (the choice of $I$ will always be clear from the context). %We use the standard notations for localizations of modules. 
Given a ring $R$, an $R$-module 
$M$, and a multiplicatively closed set $S\subset R$, we denote by $S^{-1}M$ the localisation of $M$ at $S$. If ${\mathfrak p}\subset R$ is a prime ideal we denote by $M_{\mathfrak p}$ the localisation of $M$ at the multiplicatively closed set $R\setminus {\mathfrak p}$. We use  boldface letters to denote matrices and vectors. For a module $M$, we (somewhat unusually) regard elements of $M^k$ as $k\times 1$ matrices with entries in $M$. We denote the transpose of a matrix $\mathbf A$ by $\mathbf A^{\intercal}$. 

\section*{Acknowledgements} We thank the reviewers for their comments, which helped to improve the manuscript. This research was funded, in whole or in part, by National Science Center, Poland under grant no.\ DEC-2012/07/E/ST1/00185. A CC BY or equivalent licence is applied to the AAM arising from this submission, in accordance with the grant's open access conditions.

\section{Basic notions}\label{sec:bn}

Let $R$ be a ring, $\mathbf{A}$ a $k\times l$ matrix with entries in $R$, $M$ an $R$-module, and $r\geq 1$ an integer.
\begin{definition}  We say that $\mathbf{A}$  is partition regular over $M$ for $r$ colours if for every colouring $\chi \colon M  \to \{1,\dots, r\}$ of $M$ with $r$ colours there exists a nontrivial monochromatic solution of the equation $\mathbf{A}\mathbf{m} =\mathbf{0}$ with $\mathbf{m}=(m_1,\dots, m_l)^\intercal \in M^l$, i.e.\ a solution with $$\chi(m_1)=\dots=\chi(m_l) \quad \text{ and } \quad \mathbf{m} \neq \mathbf{0}.$$ 
We say that $\mathbf{A}$ is partition regular over $M$  if  $\mathbf{A}$ is partition regular over $M$ for any (finite) number  of colours. 
\end{definition}

We begin by developing some basic properties of these notions that will often be  used in later chapters. For the rest of this section we will assume that $R$ is a ring and $\mathbf{A}$ is a matrix with entries in $R$.

Let $M$ be an $R$-module and let $N$ be its submodule. Since every colouring of $M$ induces a colouring of $N$, we see that if $\mathbf{A}$ is partition regular over $N$ for $r$ colours, then it is also partition regular over $M$ for $r$ colours. We will use this fact repeatedly without explicitly referring to it. 

Partition regularity is preserved by homomorphisms, in the following sense: let $\varphi \colon R \to S$ be a ring homomorphism and let $M$ be an $S$-module. The module $M$ can be regarded as an $R$-module via restriction of scalars (with  multiplication by $r\in R$ given by $rm=\varphi(r)m$). We denote this $R$-module by $\varphi^* M$.
 
\begin{lemma} Let $\varphi \colon R \to S$ be a ring homomorphism, $M$ an $S$-module, $\mathbf{A}$ a matrix with entries in $R$, and $r\geq 1$ an integer. Let $\varphi_*\mathbf{A}$ be the image of $\mathbf{A}$ by $\varphi$. Then $\varphi_* \mathbf{A}$ is partition regular over $M$ for $r$ colours if and only if $\mathbf{A}$ is partition regular over $\varphi^* M$ for $r$ colours.
\end{lemma}
\begin{proof} Obvious. \end{proof}

The next result is a variant of the usual  finiteness property of partition regularity. The proof uses a rather standard compactness argument.

\begin{proposition}\label{prop:compactness}
Let $M$ be an $R$-module, $\mathbf{A}$ a matrix with entries in $R$, and $r\geq 1$ an integer. If $\mathbf{A}$ is partition regular over $M$ for $r$ colours, then there exists a finite subset $F$ of $M$ such that for every colouring of $F$ with $r$ colours there exists a nontrivial monochromatic vector $\mathbf{m}$ with entries in $F$ such that  $\mathbf{A}\mathbf{m} =\mathbf{0}$.
\end{proposition}
\begin{proof}
Let $C=\{1,\dots, r\}^M$ be the space of all colourings of $M$ with $r$ colours considered as a topological space with product topology, using the discrete topology on the set $\{1,\dots, r\}$. For a finite set $F\subset M$, denote by $C_F$ the set of all colourings in $C$ that do not admit nontrivial monochromatic  solutions to the equation $\mathbf{A}\mathbf{m} =\mathbf{0}$ with entries in $F$. We will prove that $C_F=\emptyset$ for some $F$. Suppose the contrary. The sets $C_F$ are then closed and nonempty, and the family $\{C_F\}$ is closed under finite intersections. By compactness of $C$, the set $\bigcap C_F$ is nonempty, the intersection being taken over all the finite subsets of $M$. Any element of $\bigcap C_F$ is a colouring of $M$ that does not admit any nontrivial monochromatic solution  to the equation $\mathbf{A}\mathbf{m} =\mathbf{0}$ with entries  in any finite set $F\subset M$, hence neither in all of $M$. This gives a contradiction. 
\end{proof}

We will mainly use Proposition \ref{prop:compactness} in the following form.

\begin{corollary}\label{cor:compactness} Let $M$ be an $R$-module, $\mathbf{A}$ a matrix with entries in $R$, and $r\geq 1$ an integer.  \begin{enumerate} \item\label{cor:compactness1} If $\mathbf{A}$ is partition regular over $M$ for $r$ colours, then it is partition regular for $r$ colours over some finitely generated submodule of $M$.
\item \label{cor:compactness2} If $\mathbf{A}$ is partition regular over $M$, then it is partition regular over some countably generated submodule of $M$.\end{enumerate}\end{corollary}
\begin{proof} For the proof of \eqref{cor:compactness1}, take the submodule generated by a finite set $F$ given by Proposition \ref{prop:compactness}. Property \eqref{cor:compactness2} follows from \eqref{cor:compactness1}. 
\end{proof}

\begin{proposition}\label{prop:PR_under_localisation}  Let $M$ be an $R$-module, $\mathbf{A}$ a matrix with entries in $R$, and $r\geq 1$ an integer.
\begin{enumerate}
\item\label{prop:PR_under_localisation1}  Let $S$ be a multiplicative subset of $R$. Assume that $S$ does not contain zero divisors on $M$. Then $\mathbf{A}$ is partition regular over $M$ for $r$ colours if and only if it is partition regular over $S^{-1}M$ for $r$ colours.
\item\label{prop:PR_under_localisation2}  Assume that $R$ is an integral domain with fraction field $K$. Then $\mathbf{A}$ is partition regular over $R$ for $r$ colours if and only if it is partition regular over $K$ for $r$ colours.
\end{enumerate}
\end{proposition}
\begin{proof} 
For the proof of \eqref{prop:PR_under_localisation1}, assume that $\mathbf{A}$ is partition regular over $M$ for $r$ colours. Since $S$ does not contain zero divisors on $M$, the canonical map $M\rightarrow S^{-1}M$ is injective and  $\mathbf{A}$ is partition regular over $S^{-1}M$ for $r$ colours. For the opposite implication, assume that $\mathbf{A}$ is partition regular over $S^{-1}M$ for $r$ colours. By Corollary  \ref{cor:compactness} there exists a finitely generated $R$-submodule $N$ of $S^{-1}M$ such that $\mathbf{A}$ is partition regular over $N$ for $r$ colours. Choosing a finite set $\{m_1/s_1,\dots, m_t/s_t\}$ of generators of $N$, we see that $N$ is isomorphic with a submodule of $M$ via the map  $ n\mapsto s_1\cdots s_t n$. Hence $\mathbf{A}$ is partition regular over $M$ for $r$ colours.

Property \eqref{prop:PR_under_localisation2} follows immediately from \eqref{prop:PR_under_localisation1}.
\end{proof}

We end this section with a property that allows us to perform d\'evissage arguments for partition regularity.

\begin{proposition}\label{prop:PR_of_quotients} Let $M$ be an $R$-module,  $N$  its submodule, $\mathbf{A}$ a matrix with entries in $R$, and $r,s\geq 1$ integers.
\begin{enumerate}
 \item\label{prop:PR_of_quotients1} If $\mathbf{A}$ is partition regular over $M$ for $r+s$ colours, then either $\mathbf{A}$ is partition regular over $N$ for $r$ colours or $\mathbf{A}$ is partition regular over $M/N$ for $s$ colours.
\item\label{prop:PR_of_quotients2} If  $M=\bigoplus_{i=1}^t M_i$ is a direct sum of finitely many $R$-modules $M_i$, then $\mathbf{A}$ is partition regular over $M$ if and only if $\mathbf{A}$ is partition regular over some $M_i$. 
\end{enumerate}
\end{proposition}
\begin{proof}
For the proof of \eqref{prop:PR_of_quotients1}, suppose that there exist a colouring $\chi_N \colon N \to \{1,\dots,r\}$ of $N$ and a colouring $\chi_{M/N} \colon M/N \to \{1,\dots,s\}$ of $M/N$, both not admitting any nontrivial monochromatic solutions to the equation $\mathbf{A}\mathbf{m}=\mathbf{0}$ in $N$ (resp., in $M/N$). Denote by $\bar{m}$ the image of $m\in M$ in $M/N$. Consider the colouring $\chi \colon M \to \{1,\dots, r+s\}$ given by $$ \chi(m)=\begin{cases} 
\chi_N(m) \quad \text{if} \quad m \in N,\\
r+\chi_{M/N}(\bar{m}) \quad \text{if} \quad m \notin N. \end{cases}$$ It is then easy to see that the colouring $\chi$ does not admit any nontrivial monochromatic solutions to the equation $\mathbf{A}\mathbf{m}=\mathbf{0}$ in $M$.

Property \eqref{prop:PR_of_quotients2} follows immediately from \eqref{prop:PR_of_quotients1}.
\end{proof}

\section{Partition regularity over modules}\label{sec:mod}

In this section we characterise partition regularity for finitely generated modules over noetherian rings. We use the notion of an associated prime. We recall that a prime ideal $\mathfrak{p}$ of $R$ is an associated prime of an $R$-module $M$ if there exists an injective $R$-module homomorphism $R/{\mathfrak p} \hookrightarrow M$; equivalently, there exists $m\in M$ with $\mathfrak{p}=\mathrm{ann}(m)$. If $M$ is a finitely generated module over a noetherian ring $R$, then the set $\Ass\, M$ of associated prime ideals of $M$ is finite (see \cite[Theorem 3.10]{book:Eisenbud}). We say that a submodule $N$ of $M$ is $\mathfrak{p}$-primary if $\Ass\, M{/}N=\{\mathfrak{p}\}$.

\begin{theorem}\label{mainthm:modules}
Let $M$ be a finitely generated module over a noetherian ring $R$ and let
$\mathbf{A}$ be a matrix with entries in $R$. The following conditions are equivalent: \begin{enumerate} \item The matrix $\mathbf A$ is partition regular over $M$. \item There exists an associated prime $\mathfrak{p}$ of $M$ such that $\mathbf{A}$ is partition regular over $R/\mathfrak{p}$. \end{enumerate}
\end{theorem}
\begin{proof}
If $\mathfrak{p}$ is an associated prime of $M$ such that $\mathbf A$ is partition regular over $R/\mathfrak{p}$, then $R/\mathfrak{p}$ embeds into $M$ and hence $\mathbf A$ is partition regular over $M$.

Assume now that $\mathbf A$ is partition regular over $M$ and let $\mathfrak{p}_1,\dots, \mathfrak{p}_t$ be the associated primes of $M$. By primary decomposition (see \cite[Theorem 3.10]{book:Eisenbud}), there exist $\mathfrak{p}_i$-primary submodules $Q_i$ of $M$ 
 such that $\bigcap_{i=1}^{t} Q_{i}=0 $. Hence $M$ embeds via the diagonal embedding into $\bigoplus_{i=1}^t M/Q_i$ and by Proposition \ref{prop:PR_of_quotients}.\eqref{prop:PR_of_quotients2}, $\mathbf A$ is partition regular over  $M/Q_i$ for some $i\in \{1,\dots, t\}$. All zero divisors of the $R$-module $M/Q_i$ are in $\mathfrak{p}_{i}$ (see \cite[Theorem 3.1]{book:Eisenbud}), and hence  by Proposition \ref{prop:PR_under_localisation}.\eqref{prop:PR_under_localisation1},  $\mathbf A$ is partition regular over the localised module  $(M/Q_i)_{\mathfrak{p}_{i}}$.

Since $\mathfrak{m}=\mathfrak{p}_{i}R_{\mathfrak{p}_{i}}$ is the only associated prime of $(M/Q_i)_{\mathfrak{p}_{i}}$, some power $\mathfrak{m}^h$ of $\mathfrak{m}$ annihilates $(M/Q_i)_{\mathfrak{p}_{i}}$ (see \cite[Proposition 3.9]{book:Eisenbud}) and  
\[
0=\mathfrak{m}^h(M/Q_i)_{\mathfrak{p}_{i}}\subset \mathfrak{m}^{h-1}(M/Q_i)_{\mathfrak{p}_{i}} \subset \dots \subset \mathfrak{m}(M/Q_i)_{\mathfrak{p}_{i}} \subset  (M/Q_i)_{\mathfrak{p}_{i}}
\]
is a finite filtration of 
$(M/Q_i)_{\mathfrak{p}_{i}}$. Every quotient $\mathfrak{m}^i(M/Q_i)_{\mathfrak{p}_{i}}/\mathfrak{m}^{i+1}(M/Q_i)_{\mathfrak{p}_{i}}$  is a finitely dimensional vector space over the field $R_{\mathfrak{p}_{i}}/\mathfrak{p}_{i}R_{\mathfrak{p}_{i}}$ and the above filtration can be refined so that all the quotients are isomorphic with the residue field $R_{\mathfrak{p}_{i}}/\mathfrak{p}_{i}R_{\mathfrak{p}_{i}}$. By repeated use of Proposition \ref{prop:PR_of_quotients}, we get that  $\mathbf A$ is partition regular over   $R_{\mathfrak{p}_{i}}/\mathfrak{p}_{i}R_{\mathfrak{p}_{i}}$. Since $R_{\mathfrak{p}_{i}}/\mathfrak{p}_{i}R_{\mathfrak{p}_{i}}$ is the fraction field of  $R/\mathfrak{p}_i$, it follows from Proposition \ref{prop:PR_under_localisation}.\eqref{prop:PR_under_localisation2} that $\mathbf A$ is partition regular over   $R/\mathfrak{p}_i$.
\end{proof}

\section{Partition regularity over integral domains and $\mathfrak m$-colourings}\label{sec:dom}

The aim of this section is to study  partition regularity over integral domains $R$. In this case the columns condition (Definition \ref{def:cc}) and the generalised columns condition  (Definition \ref{def:gcc})  coincide as long as the integral domain $R$ is infinite. If $R$ is finite (meaning that $R$ is a finite field), the generalised columns condition is more restrictive and says that the sum of all the columns is zero.

We begin with a simple lemma saying that the columns condition does not depend on the base ring, in the following sense.
\begin{lemma}\label{lem:cc_in_fields} Let $R\subset S$ be integral domains and let $\mathbf{A}$ be a  matrix with entries in $R$. Then $\mathbf{A}$ satisfies the columns condition as a matrix with entries in $R$ if and only if it satisfies the columns condition as a matrix with entries in $S$.\end{lemma}
\begin{proof}
It is immediate that if the columns condition holds for $R$, then it also holds for $S$, and that the converse holds if $S$ is the fraction field of $R$. Thus we may assume that $R$ and $S$ are both fields and that the columns condition holds over $S$. In this case, the columns condition means that a certain system of linear equations with coefficients in $R$ has a nontrivial solution in $S$. It then follows from basic linear algebra that this system  also has a nontrivial solution in $R$. 
\end{proof}

We will generalise the construction of the colourings $c_p$, which play a crucial role in the proof of Rado's theorem. Let $p$ be a prime number. Recall that the colouring $c_p \colon \Z \to \{0,\dots,p-1\}$ is given by the formula $c_p(0)=0$ and $c_p(n) = j$ if $n\neq 0$ is of the form $n=p^k (pm+j)$ for some integers $k\geq 0$, $m\in \Z$, and $ j\in \{1,\dots,p-1\}$.

Let $R$ be a local ring and let $\mathfrak m$ be its maximal ideal. Recall that the (Krull) dimension of a ring $R$, denoted by $\dim R$, is  the supremum of the lengths $r$ of chains of prime ideals \[\mathfrak{p}_0\subsetneq \dots \subsetneq \mathfrak{p}_{r-1} \subsetneq \mathfrak{p}_r=\mathfrak{m}\] in $R$. If $R$ is a local noetherian ring,  Krull's theorem states that $\mathfrak m$ cannot be generated by fewer than $\dim R$ elements (see \cite[Corollary 10.7]{book:Eisenbud}); $R$ is called a regular local ring if it is noetherian and $\mathfrak m$ can be generated by exactly $\dim R$ elements.  Every regular local ring is an integral domain (see \cite[Corollary 10.14]{book:Eisenbud}). A regular local ring of dimension one is a discrete valuation ring (see \cite[11.1]{book:Eisenbud}). This means that for a fixed prime element $\pi$ of $R$ (which necessarily generates $\mathfrak m$), every nonzero $z$ in the fraction field $K$ of $R$ can be written uniquely as $z=u\pi^{n}$ for some unit $u \in R$ and $n\in\Z$. The induced group homomorphism $v\colon K^*\to \Z$, $v(z)=n$, is called the $\pi$-adic valuation; it does not depend on the choice of $\pi$. 

The local ring $\Z_{(p)}$, which is implicit in the definition of the $c_p$ colouring, is a discrete valuation ring, and the colouring takes values in the residue field $\Z_{(p)}/p\Z_{(p)} \cong \Z/(p)$. To generalise Rado's colourings to the case of an arbitrary finitely generated $\Z$-algebra $R$, we will need to consider regular local rings which need not be of dimension one. The role of a prime $p$ from Rado's $c_p$ colouring will be played by a maximal ideal $\mathfrak{m}$ of $R$ such that $R_{\mathfrak{m}}$ is a regular local ring and the colouring will take values in a (finite) residue field $R_{\mathfrak{m}}/\mathfrak{m}R_{\mathfrak{m}} \cong R/\mathfrak{m}$ of $R_{\mathfrak{m}}$. To define this colouring we will need to pass several times to appropriate quotients, and we will  use the standard fact that localisation commutes with taking quotients: For a ring $R$, multiplicatively closed set $S\subset R$ and ideal $I\subset R$, we have $S^{-1}R/I(S^{-1}R) \cong \overline{S}^{-1}(R/I)$, where $\overline{S}$ denotes the image of $S$ in $R/I$ (see \cite[Theorem 4.2]{Matsumura-CRT}).

Let  $R$ be  a finitely generated $\Z$-algebra. Let $\mathfrak m$ be a maximal ideal of $R$ such that $R_{\mathfrak m}$ is a regular local ring. We will now construct a finite colouring of the fraction field $K$ of $R_{\mathfrak m}.$

Choose generators $\pi_1,\dots,\pi_t$ of $\mathfrak m R_{\mathfrak m}$ with $t=\dim R_\mathfrak m$. Let
$$S_i =  R_\mathfrak m/(\pi_1,\dots,\pi_i)R_\mathfrak m\quad \text{ for } i\in\{0,\dots,t\}.$$
The rings $S_i$ are regular local rings and hence are integral domains. Let $K_i$ denote the fraction field of $S_i$ (note that $K=K_0$). We have $S_t \cong R_{\mathfrak m}/\mathfrak{m}R_{\mathfrak{m}} \cong R/{\mathfrak{m}}$. Since $\Z$ is a Jacobson ring (i.e.\ every prime ideal is an intersection of maximal ideals), we conclude from a general form of Nullstellensatz (see \cite[Theorem 4.19]{book:Eisenbud}) that $R/{\mathfrak m}$  is a finite field.

For $i\in\{0,\dots,t-1\}$, the element $\pi_{i+1}$ is a prime element of $S_i$, and hence the ring $(S_i)_{(\pi_{i+1})}$ is a discrete valuation ring with fraction field $K_i$. Consider the induced $\pi_{i+1}$-adic discrete valuation $v_{i+1} \colon K^*_i \to \Z$. Every element $z\in K^*_i$ can be written as $$z = \pi_{i+1}^{v_{i+1}(z)} z' \quad \text{ for some } z'\in ((S_i)_{(\pi_{i+1})})^*.$$ Note that the residue field of $(S_i)_{(\pi_{i+1})}$ is $$(S_i)_{(\pi_{i+1})}/ \pi_{i+1} (S_i)_{(\pi_{i+1})} \cong K_{i+1}.$$ 
 
We now construct a colouring $c_{\mathfrak m} \colon K \to R/{\mathfrak m}$. Let $x\in K$. If $x=0$, we put $c_{\mathfrak m} (x)=0$. If $x\neq 0$, we put $x_0 = x$ and we construct inductively the elements $x_1,\dots,x_t$ with $x_i \in K_{i}$  such that for $i\in\{0,\dots,t-1\}$ the element $x_{i+1} \in K_{i+1}^*$ is the image of $x_i \pi_{i+1}^{-v_{i+1}(x_i)}$ in the residue field of $(S_i)_{(\pi_{i+1})}$ under the isomorphism $(S_i)_{(\pi_{i+1})}/ \pi_{i+1} (S_i)_{(\pi_{i+1})} \cong K_{i+1}$. The element $x_t$ is a nonzero element of $K_t=R/{\mathfrak m}$. We put $c_{\mathfrak m} (x)=x_t$.

Note that the definition of the colouring $c_{\mathfrak m}$ depends not only on $\mathfrak m$, but also on the choice of generators $\pi_1,\dots,\pi_t$ of $\mathfrak m R_{\mathfrak m}$. By abuse of terminology, we refer to any such colouring  as an $\mathfrak m$-\emph{colouring}.

  \begin{remark} We briefly present an alternative description of the colouring $c_{\mathfrak m}$. Any nonzero element $x\in K$ can be (non-uniquely) written in the form 
$$x = \pi_1^{a_1+1} y_1 + \pi_1^{a_1} \pi_2^{a_2+1} y_2 +\dots + \pi_1^{a_1}\cdots \pi_{t-2}^{a_{t-2}} \pi_{t-1}^{a_{t-1}+1} y_{t-1}+ \pi_1^{a_1}\cdots \pi_{t-1}^{a_{t-1}} \pi_t^{a_t} y_{t}$$
 with $a_1,\dots,a_t \in \Z$, $y_i \in (R_{\mathfrak m})_{(\pi_1,\dots,\pi_i)}$ for $i\in\{1,\dots,t-1\}$, and $y_t\in R_{\mathfrak m}^*$. (The existence of such a representation is proved by an induction on $t$.)  Let $z$ denote the image of $y_t$ in $R_{\mathfrak m}/{\mathfrak m}R_{\mathfrak m} \cong R/{\mathfrak m}$. Then $c_{\mathfrak m}(x)=z$.
  \end{remark}

\begin{example}\mbox{ }

\begin{enumerate} 

\item Let $R=\Z$ and let $p$ be a prime number. The ring $R_{(p)}$ is a regular local ring (actually, a discrete valuation ring), and we recover Rado's colouring $c_p$ as an example of an $\mathfrak m$-colouring for $\mathfrak m =(p)$.
\item Let $R=\Z[x,y]$, let $p$ be a prime, and let $\mathfrak m = (p,x,y)$. For $f\in R$, write $$f=\sum_{(i,j)\in \N^2} f_{ij} x^i y^j.$$ The $\mathfrak m$-colouring associated to the choice of generators $\pi_1=x, \pi_2=y, \pi_3 =p$ is given by $$c_{\mathfrak m}(f) = c_p(f_{i_0 j_0}),$$ where $(i_0,j_0)$ is the lexicographically smallest element of $\N^2$ with $f_{i_0 j_0} \neq 0$.
\item In the previous example, take instead $\pi_1=p, \pi_2=x, \pi_3 =y$. Then $$c_{\mathfrak m}(f) = c_p(f_{i_1 j_1}),$$ where $(i_1,j_1)$ is the lexicographically smallest element of $\N^2$ among all the elements $f_{i_1 j_1}$  with minimal $p$-valuation.
\end{enumerate}
\end{example}

\begin{lemma}\label{oneequationdom}  Let  $R$ be  a finitely generated $\Z$-algebra, $\mathfrak m$ a maximal ideal of $R$ such that $R_{\mathfrak m}$ is a regular local ring, and $K$ the fraction field of $R_{\mathfrak m}$.  Let $a_1,\dots,a_l \in R$. If the equation $\sum_{i=1}^l a_i m_i = 0$ has a nontrivial monochromatic solution $(m_1,\dots,m_l)^{\intercal}\in K^l$ with respect to an ${\mathfrak m}$-colouring, then there exists a nonempty $I\subset \{1,\dots,l\}$ such that $$\sum_{i\in I} a_i \in \mathfrak m.$$\end{lemma}
\begin{proof} We will prove the claim by induction on $t=\dim R_\mathfrak m$. If $t=0$, then by Nullstellensatz $R= K$ is a finite field, and the fact that  $(m_1,\dots,m_l)$ is monochromatic means that $m_1,\dots, m_l$ are all equal and nonzero. It follows that $\sum_{i=1}^l a_i=0$.

If $t>0$, write $\mathfrak{m}R_{\mathfrak m} = (\pi_1,\dots,\pi_t) R_{\mathfrak m}$ with $t=\dim R_\mathfrak m$, let $c_{\mathfrak m}$ be the associated ${\mathfrak m}$-colouring, and as before denote the $\pi_1$-adic valuation on $K$ by $v_1$. Let $S=R/(\pi_1)$ and $\mathfrak n = \mathfrak{m}/(\pi_1)$. For $x\in (R_{\mathfrak m})_{(\pi_1)}$, denote by $\bar{x}$ the image of $x$ in the fraction field of $S_{\mathfrak n}$ by the quotient map. Then $\mathfrak n$ is a maximal ideal of $S$, $\mathfrak{n}S_{\mathfrak n}=(\bar{\pi}_2,\dots,\bar{\pi}_t)$, and $S_{\mathfrak n} \cong R_{\mathfrak m}/(\pi_1)$ is a regular local ring. Directly from the definition, we see that if $c_{\mathfrak n}$ is the $\mathfrak n$-colouring (with the choice of generators $\bar{\pi}_2,\dots,\bar{\pi}_t$ of $\mathfrak{n}S_{\mathfrak n}$), then for $x\in K^*$ we have $$c_{\mathfrak{m}}(x)=c_{\mathfrak n}(\overline{\pi_1^{-v_1(x)} x}).$$ 

Let now $(m_1,\dots,m_l)^{\intercal}\in K^l$  be a nontrivial monochromatic solution of the equation $\sum_{i=1}^l a_i m_i = 0$ with respect to the colouring $c_{\mathfrak m}$. Then all $m_i$ are nonzero (since $c_{\mathfrak m}(m_i)=0$ only for $m_i=0$). Put $\nu =\min_{1\leq i \leq l} v_1(m_i)$ and let $$J=\{i\in\{1,\dots,l\}\mid v_1(m_i) = \nu\}.$$ Multiplying the equation  $\sum_{i=1}^l a_i m_i = 0$ by $\pi_1^{-\nu}$ and passing to the fraction field of $S_{\mathfrak n}$, we get $$\sum_{i\in J} \bar{a}_i \overline{\pi_1^{-\nu} m_i}=0.$$ Furthermore, we have $c_{\mathfrak n}(\overline{\pi_1^{-\nu} m_i})=c_{\mathfrak{m}}(m_i)$ for $i\in J$, and hence the elements $\overline{\pi_1^{-\nu} m_i}$ are monochromatic for $i\in J$. By the induction hypothesis, there exists a nonempty subset $I\subset J$ such that $\sum_{i\in I} \bar{a}_i$ lies in $\mathfrak{n}$. Hence $\sum_{i\in I} a_i$ lies in $\mathfrak{m}$.
\end{proof}

In order to proceed, we need the following fundamental fact. 

\begin{lemma}\label{regularlocus}  Let  $R$ be an integral domain that is a finitely generated $\Z$-algebra. Then there exists a maximal ideal $\mathfrak m$ of $R$ such that $R_{\mathfrak m}$ is a regular local ring.\end{lemma}

This is a well-known result that is usually proven in a much more general context of excellent rings, introduced by Grothendieck. The ring $\Z$ is an example of an excellent ring, as is any Dedekind integral domain of characteristic zero. For the proof of Lemma \ref{regularlocus}, see, e.g.\ \cite[Corollaire 6.12.6]{bookEGAIV.2} or \cite[(32.B)]{bookMatsumuraCA}.

\begin{lemma}\label{auxlemma:domains} Let $R$ be an integral domain with fraction field $K$ and let $\mathbf{A}$ be a  matrix with entries in $R$. Let $R'$ be a subring of $K$ containing all the entries of $\mathbf{A}$ and let $K'$ denote its fraction field. If $\mathbf{A}$ is partition regular over $R$, then it is also partition regular over the field of rational functions $K'(t)$.
\end{lemma}

\begin{proof}
Since $\mathbf{A}$ is partition regular over $R$, it is also partition regular over $K$. We may regard $K$ as a $K'$-module. Fix a number of colours $r$. By Corollary \ref{cor:compactness}.\eqref{cor:compactness1},  there exists a finitely dimensional $K'$-vector space $V$ such that $\mathbf A$ is partition regular over $V$ for $r$ colours. Since $V$ is isomorphic with a~$K'$-vector subspace of $K'(t)$, we conclude that $\mathbf A$ is partition regular over $K'(t)$ for $r$ colours. This gives the claim since the number of colours  $r$ was chosen arbitrarily.
\end{proof}

\begin{theorem}\label{mainthm:domains} Let $R$ be an infinite integral domain and let $\mathbf{A}$ be a $k\times l$ matrix with entries in $R$. The following conditions are equivalent: \begin{enumerate} \item The matrix $\mathbf{A}$ is partition regular over $R$. \item  The matrix $\mathbf{A}$  satisfies the columns condition.\end{enumerate}
\end{theorem}
\begin{proof} The fact that matrices satisfying the columns condition over an infinite integral domain are partition regular follows from \cite[Theorem 2.4]{BDHL}.

For the opposite implication, assume that $\mathbf{A}$ is partition regular over $R$. We will prove that $\mathbf{A}$ satisfies the columns condition. Let $K$ be the fraction field of $R$. For two vectors $\mathbf{v},\mathbf{w} \in R^k$, we denote their standard  inner product by $(\mathbf{v},\mathbf{w})$.  Let $ \mathbf{c}_1,\dots, \mathbf{c}_l \in R^k$ denote the columns of $\mathbf{A}$. Consider the set
 $$S=\{  J \subset \{1,\dots,l\} \mid \sum_{j\in J} \mathbf{c}_j \neq \mathbf{0}\}.$$ 
 We claim that we can find a vector $\mathbf{v} \in R^k$ such that  for all $J\in S$ we have 
$$(\sum_{j\in J} \mathbf{c}_j , \mathbf{v}) \neq 0.$$
In fact, for all $J\in S$ the set of vectors in $K^k$ orthogonal to $\sum_{j\in J} \mathbf{c}_j $ is a proper vector subspace of $K^k$. Since a vector space over an infinite field is not a finite union of its proper subspaces, we can find a vector in $K^k$ that is not orthogonal to $\sum_{j\in J} \mathbf{c}_j$ for all $J\in S$. Multiplying this vector by an appropriate element of $R$, we obtain a vector in $R^k$ that has the desired property.

 For $I\subset \{1,\dots,l\}$, let $V_I$ be the $K$-vector subspace of $K^k$ spanned by $\mathbf{c}_i$ with $i\in I$ and let $$ S_I=\{  J \subset \{1,\dots,l\} \mid J\cap I =\emptyset \mbox{ and } \sum_{j\in J} \mathbf{c}_j \not\in V_I\}.$$ A similar argument shows that there exists a vector $\mathbf{v}_I \in R^k$ (which depends on $I$, but not on $J$) such that \begin{equation}\label{eqn:S_Iset} (\mathbf{c}_i, \mathbf{v}_I)=0 \mbox{ for all } i\in I  \qquad \mbox{ and }\qquad (\sum_{j\in J} \mathbf{c}_j,\mathbf{v}_I)\neq 0 \mbox{ for all } J \in S_I.\end{equation}  
 
Let $R'$ be the subring of $K$ generated by all the entries of $\mathbf{A}$, $\mathbf v$, and $\mathbf{v}_I$ for $I\subset \{1,\dots, l\} $, as well as the inverses of the elements $(\sum_{j\in J} \mathbf{c}_j,\mathbf{v})$ for $J\in S$ and $(\sum_{j\in J} \mathbf{c}_j,\mathbf{v}_I)$ for $I\subset \{1,\dots,l\}$ and $J\in S_I$. Denote the fraction field of $R'$ by $K'$. We will now prove that the matrix $\mathbf{A}$ satisfies the columns condition.

Consider the polynomial ring $R''=R'[t]$ in one variable over $R'$. The ring $R''$ is an integral domain that is finitely generated as a $\Z$-algebra. By Lemma \ref{regularlocus}, there exists a maximal ideal $\mathfrak{m}$ of $R''$ such that $R''_{\mathfrak m}$ is a regular local ring. Let $c_{\mathfrak m}$ be an $\mathfrak m$-colouring of the fraction field $K''=K'(t)$ of $R''_{\mathfrak m}$. By Lemma \ref{auxlemma:domains}, $\mathbf{A}$ is partition regular over $K''$ and so  the equation $\mathbf{A} \mathbf{m}=\mathbf{0}$ has a nontrivial monochromatic solution $\mathbf{m} \in (K'')^l$.

We first claim that there exists $I_0\subset \{1,\dots, l\}$ such that $\sum_{i\in I_0} \mathbf{c}_i =\mathbf{0}$. Taking the inner product of  $\mathbf{A} \mathbf{m}=\mathbf{0}$ with the vector $\mathbf{v}$, we get $$\sum_{i=1}^l (\mathbf{c}_i,\mathbf{v})m_{i}=0.$$ By Lemma \ref{oneequationdom}, there exists a nonempty subset $I_0 \subset \{1,\dots,l\}$ such that $\sum_{i\in I_0} (\mathbf{c}_i,\mathbf{v}) \in \mathfrak m$. This means that $$\sum_{i\in I_0} \mathbf{c}_i =\mathbf{0},$$ since otherwise we would have $I_0\in S$, and hence $(\sum_{i\in I_0} \mathbf{c}_i,\mathbf{v})$ would be invertible in $R' \subset R''$. 

We will now construct inductively nonempty subsets $I_1,\dots,I_m$ such that $\{1,\dots,l\}=I_0\cup\dots\cup I_m$ and for $t\in\{1,\dots,m\}$ we have $$I_t  \subset \{1,\dots,l\} \setminus (I_0\cup\dots\cup I_{t-1}) \qquad \text{ and } \qquad \sum_{i\in I_t} \mathbf{c}_i\in V_{t-1},$$ where $V_{t-1}$ denotes the $K$-vector space spanned by columns $\mathbf{c}_i$ with $i\in I_0\cup \dots\cup I_{t-1}$. 

Assume that the subsets $I_1,\dots,I_{t-1}$ have already been constructed, but $I_0\cup\dots\cup I_{t-1} \varsubsetneq \{1,\dots,l\}$. We will construct the set $I_t$. Let  $\mathbf{v}_{t-1}=\mathbf{v}_{I_0\cup\dots\cup I_{t-1}}$ be the vector considered in \eqref{eqn:S_Iset}. Taking the inner product of  $\mathbf{A} \mathbf{m}=\mathbf{0}$ with the vector $\mathbf{v}_{t-1}$, we get $$\sum_{i=1}^l (\mathbf{c}_i,\mathbf{v}_{t-1})m_{i}=0.$$
Since $(\mathbf{c}_i,\mathbf{v}_{t-1})=0$ for all $i\in I_0\cup\dots\cup I_{t-1}$, using once more  Lemma \ref{oneequationdom}  we get  that there exists a nonempty subset $I_t \subset \{1,\dots,l\}\setminus (I_0\cup\dots\cup I_{t-1})$ such that $\sum_{i\in I_t} (\mathbf{c}_i,\mathbf{v}_{t-1}) \in \mathfrak m$. This means that $\sum_{i\in I_t} (\mathbf{c}_i,\mathbf{v}_{t-1})$ is not invertible in $R' \subset R''$ and  hence
$$\sum_{i\in I_t} \mathbf{c}_i \in V_{t-1}.$$ 

This ends the inductive construction and shows that $\mathbf A$ satisfies the columns condition with the corresponding partition $\{1,\dots,l\}=I_{0}\cup\dots\cup I_m$.
  \end{proof}
  
  \begin{remark}  In \cite{HLS2003} it was pointed out that while in the classical version of Rado's Theorem there are several known proofs of the claim that matrices satisfying the columns condition are partition regular, there is essentially only one known proof of the opposite implication, and it uses colourings $c_p$. As a corollary of the proof, one obtains the slightly unusual statement that a matrix with integer entries is partition regular over $\Z$ if and only if it is partition regular with respect to all the colourings $c_p$. The proof of Theorem \ref{mainthm:domains} establishes the following generalisation: A matrix with entries in an integral domain $R$ that is a finitely generated $\Z$-algebra is partition regular if and only if it is partition regular with respect to all the $\mathfrak m$-colourings of $R$.  \end{remark}

\section{Rado rings}\label{sec:Rado}

In this section we study Rado rings. We recall that each matrix $\mathbf{A}$ with entries in a ring $R$ that satisfies the generalised columns condition is partition regular over $R$. A ring $R$ is called a Rado ring if the converse holds for all matrices $\mathbf{A}$ with entries in $R$. 
If $R$ is an infinite integral domain, the columns condition and the generalised columns condition coincide. If $R$ is a finite field, the generalised columns condition over $R$ is stronger. Nevertheless, in both cases we obtain the following result.

\begin{corollary}\label{cor:domainsareRado} Every integral domain is a Rado ring.
\end{corollary}
\begin{proof} If $R$ is infinite, this follows from Theorem \ref{mainthm:domains}. If $R$ is finite, we may give each element of $R$ a different colour; we then easily see that a matrix $\mathbf{A}$ with entries in $R$ is partition regular over $R$ if and only if the sum of all the columns is zero, which in this case is equivalent to the generalised columns condition.\end{proof}

In \cite{BDHL}, the only given example of a non-Rado ring was the infinite product ring $R=\prod_{i=1}^{\infty} \Z{/}n\Z$ for a non-squarefree integer $n$. This example is somewhat unsatisfactory, since the ring in question is not noetherian.

In the next subsection we will classify noetherian rings that are Rado. As a by-product, we obtain many examples of  noetherian non-Rado rings.

\subsection*{Noetherian Rado rings}

 We begin with the following lemma that will be used to show that certain rings are not Rado.

\begin{lemma}\label{lemexplmat} Let $R$ be a ring and let $b$ be an element of $R$. Consider the   $3\times 3$ matrix
\[
\mathbf{B}=\begin{pmatrix}
1 & 1 & -1 \\
0 & b & 0 \\
0 & 0 & b \end{pmatrix}
\] 
\begin{enumerate} \item\label{lemexplmat1} The matrix $\mathbf{B}$ is partition regular over $R$ if and only if $\ann(b)$ is infinite.
\item\label{lemexplmat2} The matrix $\mathbf{B}$ satisfies the generalised columns condition over $R$ if and only if there exists $d\in R$ such that $db=0$ and $d^n R$ is infinite for all $n\geq 0$.\end{enumerate}
\end{lemma}
\begin{proof} We see from the form of the matrix that $\mathbf{B}$ is partition regular over $R$ if and only if the equation $x+y-z=0$ is partition regular over $\ann(b)$.  By the main result of \cite{Deuber1975} this is equivalent to the fact that $\ann(b)$ is infinite. For the convenience of the reader, we give a sketch of an alternative direct proof. If $\ann(b)$ is finite, the equation is clearly not partition regular. If $\ann(b)$ contains (as an abelian group) elements of arbitrarily high order, then the equation is partition regular by the finite form of Schur's Theorem. Otherwise, if $\ann(b)$ is infinite, but all the elements have bounded order, then $\ann(b)$ contains for some prime $p$ a subgroup $V$  that is an $\F_p$-vector space of countable infinite dimension. Since the coefficients of the equation $x+y-z=0$ are in $\F_p$, the group $V$ can also be identified with $\F_p[X]$ regarded as an $\F_p[X]$-module. The conclusion follows from Theorem \ref{mainthm:domains}.

For the proof of \eqref{lemexplmat2}, suppose first that there exists $d\in R$ such that $db=0$ and $d^n R$ is infinite for all $n\geq 0$. Denote the columns of $\mathbf B$ by $\mathbf c_1$, $\mathbf c_2$, $\mathbf c_3$. We claim that (using the notation of Definition \ref{def:gcc}) the matrix $\mathbf{B}$ satisfies the generalised columns condition with $m=1$, $I_0=\{2,3\}$, $I_1=\{1\}$, and $d_0=d_1=d$. In fact, we only need to note that $d(\mathbf c_2 +\mathbf c_3)=\mathbf{0}$ and $d\mathbf c_1 = d\mathbf c_2 \in R\mathbf c_2 + R\mathbf c_3$. 

Conversely, suppose that $B$ satisfies the generalised columns condition with some choice of $m$, partition $\{1,2,3\}=I_0\cup \dots \cup I_m$, and elements $d_0,\dots,d_m$ satisfying the conditions of Definition \ref{def:gcc}. By looking at the top entry, we see that $I_0$ has exactly two elements (otherwise, the top entry in $\sum_{i\in I_0} \mathbf c_i$ would be a unit, and hence could not be annihilated by $d_0 \neq 0$). Hence $m=1$. 

The column $c_j$ with $j\in I_1$ satisfies $d_1\mathbf c_j \in \sum_{i\in I_0} R \mathbf c_i$. Considering the three possible choices of $I_0$, we easily compute that $d_1 b=0$. (For example, if $I_0=\{2,3\}$, then \[
d_1  \begin{pmatrix}
1 \\
0  \\
0 \end{pmatrix} \ =  a_1\begin{pmatrix}
1  \\
b  \\
0 \end{pmatrix}  \ +  a_2\begin{pmatrix}
-1  \\
0 \\
b \end{pmatrix} 
\] for some $a_1, a_2\in R$. Then $d_1=a_1-a_2$, $a_1 b = a_2 b=0$, and hence $d_1 b=(a_1-a_2)b=0$. The reasoning is analogous in the remaining two cases.) By the generalised columns condition, $d_0d_1^n R$ is infinite for all $n\geq 0$. Thus $d=d_1$ satisfies the conditions in \eqref{lemexplmat2}.
\end{proof}

\begin{theorem}\label{thm:Radonoeth} Let $R$ be a noetherian ring. The following conditions are equivalent: \begin{enumerate} \item\label{thm:Radonoeth1} $R$ is a Rado ring. \item\label{thm:Radonoeth2} For every $\mathfrak p\in \mathrm{Ass}\, R$ either $R{/}{\mathfrak p}$ is a finite field or the ring $R_{\mathfrak p}$ is a field.\end{enumerate}\end{theorem}
\begin{proof} Assume first that $R$ is a Rado ring and suppose that there exists a prime ideal $\mathfrak p\in{\Ass}\, {R}$ such that $R{/}\mathfrak p$ is infinite and $R_{\mathfrak p}$ is not a field. The latter means that $\mathfrak{p}R_{\mathfrak p} \neq 0$. The ideal $\mathfrak p$ might or might not be a minimal prime ideal; let $Q$ denote the set of minimal prime ideals of $R$ other than $\mathfrak p$. Since $R$ is noetherian, $Q$ is finite. Let $I=\{x\in \mathfrak p \mid x/1 = 0 \text{ in } R_{\mathfrak p}\}.$ Since $\mathfrak p R_{\mathfrak p} \neq 0$, we have $I\varsubsetneq \mathfrak p$. By prime avoidance (see \cite[Lemma 3.3]{book:Eisenbud}), there exists $b\in \mathfrak p$ such that $b\notin I$ and $b\notin \mathfrak q$ for all $\mathfrak q \in Q$. (We use here a variant of prime avoidance that allows for one ideal not to be prime.) We will prove that for this choice of $b$, the matrix $\mathbf B$ considered in Lemma \ref{lemexplmat} is partition regular over $R$, but does not satisfy the generalised columns condition, which contradicts the fact that $R$ is a Rado ring.

Since $\mathfrak p\in \Ass\, R$, there exists $c \in R$ such that $\mathfrak p=\ann(c)$. Note that $\ann(b)$ contains $Rc$ which as an $R$-module is isomorphic to $R{/}\mathfrak p$. Hence $\ann(b)$ is infinite, and thus by Lemma  \ref{lemexplmat}.\eqref{lemexplmat1} the matrix $\mathbf B$ is partition regular over $R$.

Suppose now that $d\in R$ is such that $db=0$. Since  $b\notin \mathfrak q$ for all $\mathfrak q \in Q$, we have $d \in \mathfrak q$ for all $\mathfrak q \in Q$. Furthermore, in the ring  $R_{\mathfrak p}$  we have $db/1 = 0$ and $b/1\neq 0$, and hence $d\in \mathfrak p$. Thus, $d$ is contained in all the minimal prime ideals of $R$, and hence is nilpotent (see \cite[Corollary 2.12]{book:Eisenbud}). By Lemma \ref{lemexplmat}.\eqref{lemexplmat2}, $\mathbf B$ does not satisfy the generalised columns condition. This ends the proof of the implication $ \eqref{thm:Radonoeth1} \Rightarrow \eqref{thm:Radonoeth2}$.

For the proof of the opposite implication, assume that for every $\mathfrak p\in \Ass\, R$ either $R{/}{\mathfrak p}$ is a finite field or the ring $R_{\mathfrak p}$ is a field, and choose a matrix $\mathbf A$ with entries in $R$ that is partition regular over $R$. We need to prove that  $\mathbf A$ satisfies the generalised columns condition. By  Theorem \ref{mainthm:modules}, there exists $\mathfrak p \in \Ass\, R$ such that $\mathbf A$ is partition regular over $R{/}\mathfrak p$. Write $\mathfrak p =\ann(d)$ for $d\in R$. Denote the columns of $\mathbf A$ by $\mathbf c_1, \dots, \mathbf c_l$.  We need to consider two cases. \begin{description}

\item[\normalfont\emph{Case 1}]  $R{/}{\mathfrak p}$ is a finite field. Since   $\mathbf A$ is partition regular over the finite field $R{/}\mathfrak p$, which is only possible if $\mathbf c_1 +\dots + \mathbf c_l =\mathbf{0} $ in $R{/}\mathfrak p$, we have $d( \mathbf c_1 +\dots+ \mathbf c_l)=\mathbf{0}$ in $R$. Hence $\mathbf A$ satisfies the generalised columns condition over $R$.

\item[\normalfont\emph{Case 2}] $R_{\mathfrak p}$ is a field and $R{/}{\mathfrak p}$ is infinite.  By Theorem \ref{mainthm:domains}, since  $\mathbf A$ is partition regular over $R{/}\mathfrak p$, there exists an integer $m\geq 0$, a partition $\{1,\dots,l\}=I_0\cup \dots\cup I_m$ and elements $d'_1,\dots,d'_m \in R \setminus \mathfrak p$  such that $\sum_{i \in I_0} \mathbf{c}_i \in \mathfrak p R^k$ and $$ d'_t \sum_{i \in I_t} \mathbf{c}_i \in \sum_{j\in I_0\cup\dots\cup I_{t-1}}\!\!\!\!\!\!\!\! R \mathbf{c}_j + \mathfrak p R^k \quad \text{ for } \quad t\in\{1,\dots,m\}.$$ Put $d_0=d$ and $d_t=dd'_t$ for $t\in\{1,\dots,m\}$. Since $\mathfrak p=\ann(d)$, in order to prove that $\mathbf A$ satisfies the generalised columns condition with the choice of $m$, partition $\{1,\dots,l\}=I_0\cup \dots\cup I_m$, and elements $d_0,\dots, d_m$, it is enough to note that $Rd_0(d_1\cdots d_m)^n$ is infinite for all $n\geq 0$. We claim that $d\notin \mathfrak p$. Indeed, since $R_{\mathfrak p}$ is a field, we would otherwise have $d/1 =0$ in $R_{\mathfrak p}$, which contradicts $\ann(d)=\mathfrak{p}$. Thus $d_t\notin \mathfrak p$ for $t\in\{0,\dots,m\}$, which implies that $\ann(d_0(d_1\cdots d_m)^n)\subset \mathfrak p$ for all $n\geq 0$. Thus $d_0(d_1\cdots d_m)^n R$ surjects onto $R{/}\mathfrak p$, and hence is infinite. This ends the proof.\qedhere \end{description}\end{proof}

\begin{corollary} Every reduced noetherian ring is Rado.\end{corollary}
\begin{proof} If $R$ is a reduced noetherian ring, then $\Ass \,R$ consists exactly of the minimal prime ideals of $R$ (this follows from \cite[Corollary 2.12 and Theorem 3.10]{book:Eisenbud}). Thus, for every $\mathfrak p\in \Ass\, R$, the ring $R_{\mathfrak p}$ is a reduced local artinian ring, hence a field. The claim follows from Theorem \ref{thm:Radonoeth}.
\end{proof}

We can now give an example of a noetherian ring that is not a Rado ring.
\begin{example} Let $p$ be a prime number. The ring $R=(\Z/p^2\Z)[X]$ is not a Rado ring.
\end{example}
\begin{proof} The only associated prime ideal of $R$ is $\mathfrak p =pR$. The claim follows from Theorem \ref{thm:Radonoeth}.\end{proof}

\subsection*{Partition regularity over the ring $R=\prod_{i\in I} \Z/n\Z$}

In \cite{BDHL} the authors considered the problem of partition regularity of linear equations over the  product ring $R=\prod_{i=1}^{\infty} \Z/n\Z$. In particular, they showed that $R$ is not a Rado ring  if and only if  the ring $\Z/n\Z$ contains a nilpotent element. They also characterised partition regularity of single equations over the ring $R=\prod_{i=1}^{\infty} \Z/4\Z$.  

We will give a general characterisation of partition regularity for matrices over the ring $R=\prod_{i\in I} \Z/n\Z$. We first treat the case when $n=p$ is a prime, and then deduce from it the general case.

\begin{proposition}\label{thm:characterisation_of_products} Let $I$ be any set, $n$ a positive integer, and  $R=\prod_{i\in I} \Z{/}n\Z$. Let $\mathbf{A}$ be a $k\times l$ matrix with entries in $R$ and write $\mathbf A=\prod_{i\in I} \mathbf{A}_i$  with  matrices $\mathbf{A}_i$ having entries in $\Z/n\Z$. The following conditions are equivalent:
\begin{enumerate}
\item  The matrix $\mathbf{A}$ is partition regular over $R$.
\item There is a prime $p$ dividing $n$ such that either for some $i\in I$  the matrix $\mathbf{A}_i \bmod p$ satisfies the generalised columns condition or for infinitely many $i\in I$ the matrix $\mathbf{A}_i \bmod p$ satisfies the columns condition.
\end{enumerate}
\end{proposition}
\begin{proof}
We begin by treating the case when $n=p$ is a prime number. Each matrix $\mathbf{A}_i$ lies in the set $\mathrm{M}_{k\times l} (\F_p)$ of $k\times l$ matrices over $\F_p$, and hence may take only finitely many possible values. For $\mathbf B \in \mathrm{M}_{k\times l} (\F_p)$, let $I_\mathbf{B} \subset I$ denote the set of $i\in I$ such that $\mathbf{A}_i=\mathbf{B}$. We decompose the ring $R$ as a finite product $ R = \prod_{\mathbf{B}\in \mathrm{M}_{k\times l} (\F_p)} R_\mathbf{B}$ of rings $R_\mathbf{B}=\prod_{i\in I_\mathbf{B}} \F_p$. We see from Proposition \ref{prop:PR_of_quotients}.\eqref{prop:PR_of_quotients2} that $\mathbf{A}$ is partition regular over $R$ if and only if $\mathbf{B}$ (regarded as a matrix with the same entries on each coordinate) is partition regular over $R_{\mathbf B}$ for some $\mathbf B \in \mathrm{M}_{k\times l} (\F_p)$. Since $\mathbf B$ has entries in $\F_p \subset R_\mathbf{B}$, we may forget about the ring structure on $R_\mathbf{B}$, and regard it instead as an $\F_p$-vector space.

If $R_\mathbf{B}$ is finite, then $\mathbf{B}$ is partition regular over $R_\mathbf{B}$ if and only it satisfies the generalised columns condition (i.e.\ its columns add up to zero). Now assume that $R_\mathbf{B}$ is infinite. We will show that $\mathbf{B}$  is partition regular over $R_\mathbf{B}$ if and only if it satisfies the columns condition over $\F_p$. By Corollary \ref{cor:compactness}.\eqref{cor:compactness2},  partition regularity over $R_\mathbf{B}$  is equivalent to partition regularity over an $\F_p$-vector space of countable infinite dimension that can be chosen to be the ring of polynomials $\F_p[X]$. Regarding now $\mathbf{B}$ as a matrix with coefficients in  $\F_p[X]$, we conclude from Theorem \ref{mainthm:domains} that $\mathbf{B}$ is partition regular over  $\F_p[X]$ if and only if it satisfies the columns condition over $\F_p[X]$. By Lemma \ref{lem:cc_in_fields}  this is equivalent to $\mathbf{B}$ satisfying the columns condition over $\F_p$. This ends the proof in the case when $n=p$ is a prime number.

In the general case, consider the prime decomposition $n=p_1^{\alpha_{1}}\cdots p_t^{\alpha_{t}}$  of $n$, and write $R$ as  $$R= \prod_{i\in I} \Z/p_1^{\alpha_{1}}\Z \times \cdots \times\prod_{i\in I} \Z/p_t^{\alpha_{t}}\Z.$$ By Proposition \ref{prop:PR_of_quotients}.\eqref{prop:PR_of_quotients2} the matrix  $\mathbf{A}$ is partition regular over $R$ if and only if  $\mathbf{A}$ is partition regular over the ring $\prod_{i\in I} \Z/p_j^{\alpha_{j}}\Z$ for some $j\in \{1,\dots,t\}$. Consider the filtration 
\[
0\subset p_j^{\alpha_{j}-1}\Z/p_j^{\alpha_{j}}\Z\subset    \dots \subset p_j\Z/p_j^{\alpha_{j}}\Z\subset  \Z/p_j^{\alpha_{j}}\Z
\] 
of $\Z/p_j^{\alpha_{j}}\Z$ with quotients isomorphic to $\F_{p_{j}}$. Using Proposition \ref{prop:PR_of_quotients}.\eqref{prop:PR_of_quotients1},  we reduce the problem  to partition regularity over  the ring $\prod_{i\in I} \F_{p_{j}}$. This ends the proof.
\end{proof}

As an easy corollary of Proposition \ref{thm:characterisation_of_products}, we can prove that for single equations over the ring $ \prod_{i\in I} \Z/n\Z$ partition regularity is equivalent to the generalised columns condition. This was already proven in \cite{BDHL} for $n=4$ and $I$ countable.

\begin{corollary}\label{cor:answers_to_natural_questions_act1}
Let $n$ be an integer, $R=\prod_{i\in I} \Z/n\Z$, and $a_1,\dots, a_l\in R$. The following conditions are equivalent: \begin{enumerate} \item\label{cor:answers_to_natural_questions_act11} The  equation $a_1x_1+\dots + a_lx_l=0$ is partition regular over $R$. \item\label{cor:answers_to_natural_questions_act12} The matrix $(a_1,\dots,a_l)$  satisfies the generalised columns condition.\end{enumerate}
\end{corollary}
\begin{proof}
It is sufficient to prove that \eqref{cor:answers_to_natural_questions_act11} implies \eqref{cor:answers_to_natural_questions_act12}. Let $\mathbf{A}=\prod_{i\in I} \mathbf{A}_i$ denote the $1\times l$ matrix $\mathbf{A}=(a_1,\dots,a_l)$ and assume that $\mathbf{A}$ is partition regular over $R$. By Proposition \ref{thm:characterisation_of_products} there exists a prime number $p$ dividing $n$ such that one of the following two cases holds.

\begin{description}

\item[\normalfont\emph{Case 1}]  $\mathbf{A}_i \bmod p$ satisfies the generalised columns condition for some $i\in I$. In this case  $\sum_{j=1}^l a_j =0$ is a zero divisor in $R$. Then $\mathbf A$ satisfies the generalised columns condition with $m=0$.

\item[\normalfont\emph{Case 2}]  There exists a matrix $\mathbf{B}$ such that $\mathbf{A}_i=\mathbf{B}$ for $i\in I_\mathbf{B}$ with $I_\mathbf{B}\subset I$ infinite and $\mathbf{B} \bmod p$ satisfies the columns condition.  Write the matrix $\mathbf{B}$ as $(b_1,\dots,b_l)$ with $b_j \in \Z/n\Z$.  We may assume that $\mathbf{B} \bmod p$ does not satisfy the generalised columns condition (otherwise, we are in Case 1). Thus there exists $\emptyset\varsubsetneq  J \varsubsetneq \{1,\dots,l\}$ such that $\sum_{j\in J} b_j \bmod p =\mathbf{0}$ and $b_{j_0}\bmod p \neq 0$ for some $j_0\in J$. Write $n=p^{\alpha}n'$ with $p{\nmid} n'$ and let $e=(e_i)_{i\in I}\in R$ be the element with $e_i=1$ if $i\in I_\mathbf{B}$ and $e_i=0$ otherwise. Then it is easy to see that $\mathbf A$ satisfies the generalised columns condition with $m=1$, $I_0=J$, $I_1=\{1,\dots,l\}\setminus J$, $d_0=\frac{n}{p} e$ and $d_1=n'e$.\qedhere \end{description}
\end{proof}

In \cite{BDHL} the authors showed that the product $\prod_{i=1}^{\infty} \Z/np^2\Z$ is not Rado by constructing a $(p+1)\times (p+1)$ matrix that is partition regular but does not satisfy the generalised columns condition. They further asked if the minimal number of rows of such a matrix is $p+1$ \cite[p.\ 83]{BDHL}. By Corollary \ref{cor:answers_to_natural_questions_act1} we see that we cannot find such a matrix with only one row. We will show that the minimal number of rows is always at most three, and is two if $p\geq 5$. 

\begin{corollary}\label{cor:answers_to_natural_questions_act2}
Let $n$ be a positive integer, $p$ a prime number, and $I$ an infinite set. Let $R=\prod_{i\in I}\Z/np^2\Z$. There exists a $3\times 3$ matrix $\mathbf{B}$ with entries in $R$ which is partition regular over $R$, but does not satisfy the generalised columns condition.
\end{corollary}
\begin{proof} This follows immediately from Lemma \ref{lemexplmat} by taking the matrix $\mathbf B$ corresponding to $b=p$ (identified with the element $(p)_{i\in I}$ with $p$ on each coordinate).\qedhere
\end{proof}

\begin{corollary}\label{cor:answers_to_natural_questions_act3}
Let $n$ be a positive integer, $p\geq 5$ a prime number, and $I$ an infinite set. Let $R=\prod_{i\in I}\Z/np^2\Z$. There exists a $2\times 3$ matrix $\mathbf{B}$ with entries in $R$ which is partition regular over $R$, but does not satisfy the generalised columns condition.

\end{corollary}
\begin{proof} The proof is analogous to the proof of Lemma \ref{lemexplmat}, using instead the $2\times 3$ matrix 
\[
\mathbf{B}= \begin{pmatrix}
1 & p-1 & 2 \\
0 & 0 & p  \end{pmatrix} 
.\qedhere \] 
\end{proof}
 
It seems that for $p\in\{2,3\}$ there does not exist a matrix with two rows that is partition regular over $R$ but does not satisfy the generalised columns condition. A proof of this fact would however involve a lengthy case-by-case analysis and we will not attempt it.

\section{Nonhomogeneous equations}\label{sec:nonhom}

In this section we investigate the problem of partition regularity of nonhomogeneous equations over arbitrary modules. Let $R$ be a ring and let $M$ be an $R$-module. Let $\mathbf{A}$ be a $k\times l$ matrix with entries in $R$ and let $\mathbf{b}\in M^k$ be a vector. We say that the equation $\mathbf{Am}=\mathbf{b}$ is partition regular over $M$ if for any finite colouring of $M$ there exists a solution $\mathbf{m}=(m_1,\dots,m_l)^{\intercal}\in M^l$ of this equation with  $m_1,\dots, m_l$ monochromatic and not all $m_i$ zero. (The latter condition is automatic if $\mathbf{b}\neq \mathbf{0}$.)

A complete characterisation of partition regularity of nonhomogeneous equations over the ring of integers was done by Rado \cite{Rado1933}.  It states that a nonhomogeneous equation $\mathbf{Am}=\mathbf{b}$ with $\mathbf{A}, \mathbf{b}$ having integer entries and $\mathbf{b}\neq \mathbf{0}$ is partition regular over the integers if and only if it has a \emph{constant solution}, i.e.\ if there exists a vector $\mathbf m =(m,\dots,m)^{\intercal} \in \Z^l$ with all entries equal and such that $\mathbf{Am}=\mathbf{b}$.  In \cite[Theorem 4.2]{BDHL}, this characterisation was extended to a rather restricted class of integral domains (more precisely, to integral domains with at least one nonzero nonunit such that $R/mR$ is finite for each $m\in R\setminus \{0\}$). 

In this section we will generalise this result to a much wider class of rings. We will also study the problem more generally for modules. We first study the case of a single equation. Here, we replace the use of \cite[Lemma 4.1]{BDHL} (which is only proved under the above restrictive assumptions) with Theorem \ref{lem:nonhomogenous-one_eq} below which is obtained using a result of Straus \cite{Straus}. We then use the case of a single equation to describe partition regularity for systems of equations. The idea to derive the general case from the case of a single equation is standard and  due to Rado \cite{Rado1943}. It allows us to immediately obtain the desired result if $R$ is an integral domain. In general, this approach might not work, and we quantify the obstructions by introducing certain modules $H_R(I,M)$ (see Definition \ref{defHmodules}). We then study conditions under which this obstruction vanishes.

We recall the theorem of Straus.

\begin{theorem}[\cite{Straus}]\label{lemmaStraus} Let $G$ be an abelian group,  let  $f_1,\dots,f_l \colon G\to G$ be any mappings (not necessarily homomorphisms), and let $b\in G$ be a nonzero element. Then there exists a finite colouring $\chi$ of $G$ such that the nonhomogeneous equation $$\sum_{i=1}^l\left(f_i(x_i)-f_i(y_i)\right) = b$$
has no solutions $x_i, y_i$ with $\chi(x_i) = \chi(y_i)$ for $i = 1,\dots, l$.\end{theorem}

A characterisation of partition regularity of a single equation can be easily derived from this result.

\begin{theorem}\label{lem:nonhomogenous-one_eq} Let $R$ be a ring and let $M$ be an $R$-module. Let $a_1,\dots, a_l\in R$, let $b\in M$ be nonzero, and write $a=\sum_{i=1}^l a_i$. The following conditions are equivalent: \begin{enumerate} \item\label{lem:nonhomogenous-one_eq1} The equation $\sum_{i=1}^l a_i m_i=b$ is partition regular over $M$. \item\label{lem:nonhomogenous-one_eq2} The equation $\sum_{i=1}^l a_i m_i=b$ has a constant solution in $M$. \item\label{lem:nonhomogenous-one_eq3} $b\in aM$.\end{enumerate}
\end{theorem}
\begin{proof} The equivalence of \eqref{lem:nonhomogenous-one_eq2} and \eqref{lem:nonhomogenous-one_eq3} is trivial, and so is the fact that both these conditions imply \eqref{lem:nonhomogenous-one_eq1}.

We will now prove that \eqref{lem:nonhomogenous-one_eq1} implies \eqref{lem:nonhomogenous-one_eq3}. Suppose for the sake of contradiction that the equation $\sum_{i=1}^l a_i m_i=b$ is partition regular over $M$, but $b\not\in aM$. Passing to the quotient module $M/aM$ over the ring $R/aR$, we get that the equation  $\sum_{i=1}^l \bar{a}_im_i=\bar{b}$ is partition regular over $M/aM$ and $\bar{b}\neq 0$. Consider the maps $f_i\colon M/aM \to M/aM$ given by $f_i( m)=a_i m$. Applying Lemma \ref{lemmaStraus}, we obtain a colouring $\chi$ of $M/aM$ such that the equation $$\sum_{i=1}^l\left(f_i(x_i)-f_i(y_i)\right) = \bar{b}$$ has no solutions $x_i, y_i \in M/aM$ with $\chi(x_i) = \chi(y_i)$ for $i \in \{ 1,\dots, l\}$.

Since the equation $\sum_{i=1}^l \bar{a}_im_i=\bar{b}$ is partition regular over $M/aM$, it has a solution $(m_1,\dots, m_l)$ that is monochromatic with respect to $\chi$. Since $\bar{a}=\sum_{i=1}^l\bar{a}_i=0$, we can rewrite this as $$\sum_{i=1}^l(\bar{a}_i m_i-\bar{a}_i m_1)=\bar{b},$$ which contradicts Straus' result. \end{proof}

In order to study partition regularity for more general nonhomogeneous equations, we need to consider module homomorphism with very special properties.

\begin{definition}\label{defHmodules} Let $R$ be a ring, $I$ an ideal of $R$, and $M$ an $R$-module. Denote $$Z_R(I,M)=\{\varphi\in \Hom_R(I,M) \mid  \varphi(t)\in tM \text{ for all } t\in I\}.$$ This is an $R$-submodule of the module $\Hom_R(I,M)$ of all homomorphisms from $I$ to $M$. We call a homomorphism $\varphi \in Z_R(I,M)$ \emph{principal} if there exists $m\in M$ such that $\varphi(t)=tm$ for all $t\in I$. We denote the submodule of principal homomorphism by $B_R(I,M)$ and the quotient module by $H_R(I,M)=Z_R(I,M)/B_R(I,M)$.\end{definition}

The construction of $H_R(I,M)$ is functorial in $M$, in the sense that a homomorphism $f\colon M\to N$ induces by composition with $f$ a homomorphism $H_R(I,M) \to H_R(I,N)$. If $R$ is noetherian and $M$ is a finitely generated $R$-module, then the modules $H_R(I,M)$ are finitely generated, being subquotients of $\Hom_R(I,M)$.

Our aim is to find sufficient conditions for the module $H_R(I,M)$ to vanish, or---equivalently---for every homomorphism $\varphi\in Z_R(I,M)$ to be principal. This will allow us to conclude that certain systems of equations are not partition regular.

\begin{theorem}\label{thmnonhomsyst} Let $R$ be a ring and let $M$ be an $R$-module. Let $\mathbf{A}=(a_{ij})$ be a $k\times l$ matrix with entries in $R$ and let $\mathbf{b}\in M^k$ be nonzero. Write $a_i =\sum_{j=1}^l a_{ij}$ and denote by $I$ the ideal $I=(a_1,\dots,a_k)R$. Assume that $H_R(I,M)=0$. The following conditions are equivalent: \begin{enumerate} \item\label{thmnonhomsyst1} The equation $\mathbf{Am}=\mathbf{b}$ is partition regular over $M$. \item\label{thmnonhomsyst2} The equation  $\mathbf{Am}=\mathbf{b}$  has a constant solution in $M$. \end{enumerate}\end{theorem}\begin{proof} It is obvious that \eqref{thmnonhomsyst2} implies \eqref{thmnonhomsyst1}. For the opposite implication, assume that $\mathbf{Am}=\mathbf{b}$ is partition regular over $M$. For any vector $\mathbf{r}=(r_1,\dots,r_k)^{\intercal}\in R^k$, the single equation $\mathbf{r}^{\intercal}\mathbf{Am}=\mathbf{r}^{\intercal}\mathbf{b}$ obtained by taking a linear combination of rows of $\mathbf{Am}$ with coefficients from $\mathbf{r}$ is still partition regular. Applying to this equation Theorem \ref{lem:nonhomogenous-one_eq}, we conclude that \begin{equation}\label{eqnsysinsub} \mathbf{r}^{\intercal}\mathbf{b}=\sum_{i=1}^k r_i b_i \in \left(\sum_{i=1}^k r_i a_i\right)M.\end{equation}
(The proposition can only be applied if $\mathbf{r}^{\intercal}\mathbf{b}\neq 0$, but the conclusion is obvious otherwise.) Define a map $\varphi\colon I \to M$ by putting $\varphi(\sum_{i=1}^k r_i a_i) = \sum_{i=1}^k r_i b_i$. This is well-defined, since if $\sum_{i=1}^k r_i a_i=\sum_{i=1}^k r_i' a_i$ for some $ r_i, r_i'\in R$, then applying \eqref{eqnsysinsub} to $\mathbf{r}=(r_1-r_1',\dots,r_k-r_k')^{\intercal}$, we get $\sum_{i=1}^k r_i b_i=\sum_{i=1}^k r_i' b_i$. Applying   \eqref{eqnsysinsub} again, we get that $\varphi \in Z_R(I,M)$. Since $H_R(I,M)=0$, there exists $m\in M$ with $\varphi(t)=tm$ for all $t\in I$. Taking $t\in\{a_1,\dots,a_k\}$ shows that the constant vector $\mathbf{m}=(m,\dots,m)^{\intercal}\in M^l$ is a solution of $\mathbf{Am}=\mathbf{b}$.\end{proof}

We begin our study of $H_R(I,M)$  modules with a proposition that gathers a few of their simple properties.

\begin{proposition}\label{prop:Hmodbasic}  Let $R$ be a ring, $I$ an ideal of $R$, and let $M$ and $\{M_{\lambda}\}_{\lambda \in \Lambda}$ be $R$-modules. \begin{enumerate} \item \label{prop:Hmodbasicprod}$H_R(I,\prod_{\lambda \in \Lambda} M_\lambda) = \prod_{\lambda \in \Lambda} H_R(I,M_{\lambda})$.
 \item \label{prop:Hmodbasicsum} Assume that $I$ is finitely generated. Then $H_R(I,\bigoplus_{\lambda \in \Lambda} M_\lambda) = \bigoplus_{\lambda \in \Lambda} H_R(I,M_{\lambda})$.
 \item \label{prop:Hmodbasicann} If $J=\mathrm{ann}(M)$, then  $H_{R}(I,M)=H_{R/J}((I+J)/J,M)$.
 \item \label{prop:Hmodbasicmult}Let $S\subset R$ be a multiplicative set. If $I$ is finitely generated, then $S^{-1} H_R(I,M)$ is a submodule of $ H_{S^{-1}R}(S^{-1}I,S^{-1}M).$\end{enumerate} \end{proposition}
 \begin{proof}  \eqref{prop:Hmodbasicprod}, \eqref{prop:Hmodbasicsum} Immediate.
 
\eqref{prop:Hmodbasicann} It follows from the definition of $Z_R(I,M)$ that every $\varphi\in Z_R(I,M)$ factors through $I/(I\cap J) \cong (I+J)/J$. The equality follows  easily from this.

 \eqref{prop:Hmodbasicmult} Any homomorphism $\varphi\in\Hom_R(I,M) $ induces by localisation a homomorphism  $\varphi_S\in \Hom_{S^{-1}R}(S^{-1}I,S^{-1}M)$, giving rise to a map $$S^{-1} \Hom_R(I,M) \to \Hom_{S^{-1}R}(S^{-1}I,S^{-1}M).$$ This maps $S^{-1} Z_R(I,M)$ to $ Z_{S^{-1}R}(S^{-1}I,S^{-1}M)$ and $S^{-1} B_R(I,M)$ to $ B_{S^{-1}R}(S^{-1}I,S^{-1}M)$, thus inducing a map $\Phi\colon S^{-1} H_R(I,M) \to H_{S^{-1}R}(S^{-1}I,S^{-1}M).$ To prove that $\Phi$ is injective, we need to show that if the image of some $\varphi \in Z_R(I,M)$ is a principal homomorphism $\varphi_S \in B_{S^{-1}R}(S^{-1}I,S^{-1}M)$, then there exists some $s\in S$ such that $s\varphi$ is a principal homomorphism.

Suppose that $\varphi$ is as above. Since $\varphi_S$ is principal, there exists $m/s \in S^{-1}M$ such that $\varphi(t)/1=tm/s$ in $S^{-1}M$ for all $t\in I$. This means that for all $t\in I$ there exists $s_t\in S$ such that $s_t(s\varphi(t)-tm)=0$. Since the ideal $I$ is finitely generated, say by $t_1,\dots,t_k$, we can assume that $s_t$ is independent of $t$ by replacing $s_t$ with $s'=s_{t_1}\cdots s_{t_k}$. This shows that $s'(s\varphi(t)-tm)=0$ for all $t\in I$ and hence $s's\varphi$ is a principal homomorphism, finishing the proof.
 \end{proof}

The following proposition gives some sufficient conditions for the modules $H_R(I,M)$ to vanish.
\begin{proposition}\label{prop:Hmodules} Let $R$ be a ring, $I$ an ideal of $R$, and $M$ an $R$-module. \begin{enumerate} \item\label{prop:Hmodulesprin} If $I$ is principal, then $H_R(I,M)=0$. \item \label{prop:Hmodulesdom} If $R$ is an integral domain and $M$ is a torsion-free module, then $H_R(I,M)=0$.
 \item \label{prop:HmodulesDed} If $R$ is a Dedekind integral domain, then $H_R(I,M)=0$.
\item \label{prop:Hmodulesred} If $R$ is a reduced ring with finitely many minimal prime ideals (so in particular a reduced noetherian ring), then $H_R(I,R)=0$.
 \end{enumerate}\end{proposition}
 \begin{proof} 
 \eqref{prop:Hmodulesprin} Assume that $I=sR$ is a principal ideal and $\varphi\in Z_R(I,M)$. Write $\varphi(s)=sm$ with $m\in M$. Since $\varphi$ is an $R$-module homomorphism, it is clear that $\varphi(rs)=rsm$ for all $r\in R$ and hence $\varphi$ is principal.

 \eqref{prop:Hmodulesdom} Assume that $R$ is an integral domain. The claim is clear if $I=0$, so assume that $I\neq 0$. Choose some nonzero $s\in I$ and write $\varphi(s)=sm$ for some $m\in M$. For every $t\in I$ we have $$s\varphi(t)=\varphi(st)=t\varphi(s)=stm.$$ Since $M$ is torsion-free, we get $\varphi(t)=tm$. Therefore $\varphi$ is principal.

 \eqref{prop:HmodulesDed} If $R$ is a discrete valuation ring, the claim follows from \eqref{prop:Hmodulesprin} as every ideal of $R$ is  principal. In general, the proof follows from a standard localisation argument using Proposition \ref{prop:Hmodbasic}.\eqref{prop:Hmodbasicmult}; indeed, for every maximal ideal $\mathfrak p$  of $R$, $R_{\mathfrak p}$ is a discrete valuation ring, so we get $H_R(I,M)_{\mathfrak p}\subset H_{R_{\mathfrak p}}(I_{\mathfrak p},M_{\mathfrak p})=0$, and hence $H_R(I,M)=0$ by \cite[Lemma 2.8]{book:Eisenbud}. 

\eqref{prop:Hmodulesred} Choose a map $\varphi \in Z_R(I,R)$. We will show that $\varphi$ is principal. Let $\mathfrak p_1, \dots, \mathfrak p_n$ be the minimal prime ideals of $R$. By Proposition \ref{prop:Hmodules}.\eqref{prop:Hmodulesdom} and Proposition \ref{prop:Hmodbasic}.\eqref{prop:Hmodbasicann},  the modules $H_R(I,R/\mathfrak p_i)=H_{R{/}\mathfrak p_i}((I+\mathfrak p_i){/}\mathfrak p_i,R/\mathfrak p_i)$ vanish for all $i$.   Composing $\varphi$ with the canonical projections $R\to R/\mathfrak{p_i}$ and using the fact that $H_R(I,R/\mathfrak p_i)=0$, we obtain elements $r_i\in R$ such that \begin{equation}\label{eq:Hmodred} \varphi(t)-tr_i \in \mathfrak p_i \quad  \quad \text{ for  all }  t\in I.\end{equation}

 After renumbering the prime ideals $\mathfrak p_1, \dots, \mathfrak p_n$, we may assume that $I$ is contained precisely in $\mathfrak p_1, \dots, \mathfrak p_m$ for some $0\leq m\leq n$. By prime avoidance (see \cite[Lemma 3.3]{book:Eisenbud}), there exists an element $t_0 \in I \setminus  \bigcup_{i=m+1}^n \mathfrak p_i$. By definition of $\varphi$, there exists $r\in R$ such that $\varphi(t_0)=t_0r$. Applying \eqref{eq:Hmodred} for $t=t_0$, we get $t_0r-t_0r_i \in \mathfrak p_i$ for all $i$, and hence $r-r_i \in \mathfrak p_i$ for $ m+1\leq i \leq n$. 

We claim that $\varphi(t)=tr$ for all $t\in I$. Since the ring $R$ is reduced, $\mathfrak p_1 \cap \dots \cap \mathfrak p_n = \mathrm{nil}(R) =0$ (see \cite[Corollary 2.12]{book:Eisenbud}), and hence it is enough to check that $\varphi(t)-tr \in \mathfrak p_i$ for all $i$. This is true for $i \leq m$ since in this case both $\varphi(t)$ and $tr$ lie in $\mathfrak p_i$, and for $i \geq m+1$ since in this case $\varphi(t)-tr_i$ and $r-r_i \in \mathfrak p_i$.
 \end{proof}

In the remaining part of this section we give examples when modules $H_R(I,M)$ do not vanish. Such examples are rather easy to construct if $M$ is not assumed to be finitely generated, but are more involved otherwise. We give three examples of triples $(R,I,M)$ for which $H_R(I,M)\neq 0$: \begin{enumerate} \item[(a)] when $R$ is a noetherian integral domain and $M$ is a torsion and not finitely generated $R$-module; \item[(b)] when $R$ is a local artinian ring and $M=R$; \item[(c)] when $R$ is a noetherian integral domain and $M$ is a finitely generated (torsion) $R$-module.\end{enumerate}

\begin{example} Let $k$ be a field, $R=k[X,Y]$, $I=(X,Y)$, and $M=k(X,Y)/k[X,Y]$.  Then $H_R(I,M)\neq 0$. \end{example}
\begin{proof} For $f\in k(X,Y)$, we denote by $\overline{f}$ its image by the quotient map $k(X,Y)\to k(X,Y)/k[X,Y]$. Consider the unique $R$-linear map $\varphi\colon I \to M$ such that $\varphi(X)=\overline{Y^{-1}}, \varphi(Y)=\overline{YX^{-1}}$. Since $tM=M$ for all nonzero $t\in R$, it is clear that $\varphi \in Z_R(I,M)$. We claim that  $\varphi \not\in B_R(I,M)$. Indeed, suppose otherwise and write $\varphi(t)=t\overline{f}$ for some $f\in k(X,Y)$. Comparing the values of $\varphi(X)$ and $\varphi(Y)$, we get $Xf-Y^{-1}=g$ and $Yf-YX^{-1}=h$ for some $g,h \in k[X,Y]$. Thus $Yg-Xh=Y-1$, which gives a contradiction. \end{proof}

\begin{example}\label{examplemod1} Let $k$ be a field and consider the ring \begin{equation}\label{ringcount} R=k[X,Y,Z,W]/\left((X,Y,Z,W)^3+(Z^2, ZW, W^2, YW, YZ-XW)\right).\end{equation} We denote by $x,y,z,w\in R$ the images of $X,Y,Z,W$ in $R$. Let $I=(x,y)R$. Then $H_R(I,R)\neq 0$.
\end{example}
\begin{proof} The ring $R$ is a $10$-dimensional $k$-algebra with basis $1, x, y, z, w, x^2, xy, y^2, xz, xw$. The ideal $I$ is a vector subspace with basis $\mathcal B$ given by $x, y, x^2, xy, y^2, xz, xw$. Define the map $\varphi\in \Hom_R(I,M)$ on the basis $\mathcal B$ by mapping $x$ to $\varphi(x)=xz$ and mapping the remaining elements of $\mathcal B$ to $0$. It is easy to check that this is an $R$-module homomorphism. We claim that $\varphi\in Z_R(I,M)$. Any $t\in I$ can be written as $t =\lambda x + s$ with $\lambda \in k$ and $s\in J=(y, x^2, xz, xw)R$. We claim that $\varphi(t)\in tR$. This is clear if $\lambda = 0$, since then $\varphi(t)=0$. On the other hand if $\lambda \neq 0$,  we may rewrite $t$ in the form $t=ux+vy$ for a unit $u\in R^*$ and $v\in R$. Then $$\varphi(t)=\lambda x z = (ux+vy)(z-u^{-1}vw) $$ lies in the ideal $tR$.

We now claim that $\varphi \notin B_R(I,M)$. Suppose the contrary. Then there exists $r\in R$ such that $\varphi(t)=tr$ for all $t\in I$. Putting $t=x$ and $t=y$ gives $xr=xz$ and $yr=0$. Writing $r$ in the basis of $R$ and using these equalities gives a contradiction.\end{proof}

\begin{example}\label{examplemod2} Let $R=k[X,Y,Z,W]$, $I=(X,Y)R$, and let $M$ denote the quotient ring considered in \eqref{ringcount}, now regarded as an $R$-module. Then $H_R(I,M)\neq 0$.\end{example}
\begin{proof} Follows from Example \ref{examplemod1} and Proposition \ref{prop:Hmodbasic}.\eqref{prop:Hmodbasicann}.\end{proof}

\begin{theorem}\label{corprinhom} Let $R$ be a ring and let $M$ be an $R$-module. Assume that one of the following assumptions holds: \begin{enumerate} \item[(a)] either $R$ is an integral domain and $M$ is a torsion-free module; or \item[(b)] $R$ is a Dedekind integral domain; or \item[(c)] $R$ is a reduced ring with finitely many minimal prime ideals and $M=R$.\end{enumerate}
 Let $\mathbf{A}$ be a $k\times l$ matrix with entries in $R$ and let $\mathbf{b}\in M^k$ be nonzero. The following conditions are equivalent: \begin{enumerate} \item\label{thmnonhomsyst1} The equation $\mathbf{Am}=\mathbf{b}$ is partition regular over $M$. \item\label{thmnonhomsyst2} The equation  $\mathbf{Am}=\mathbf{b}$  has a constant solution in $M$.\end{enumerate}\end{theorem}
 \begin{proof} Follows immediately from Theorem \ref{thmnonhomsyst}  and Proposition \ref{prop:Hmodules}.\end{proof}

Note that vanishing of the modules $H_R(I,M)$ is only needed for the particular method of the proof, and not necessarily for the claim of Theorem \ref{corprinhom}. For example, if in Example \ref{examplemod1} the field $k$ (and hence the ring $R$) is finite, then the claim of Theorem \ref{corprinhom} certainly holds (since we may give each element of $R$ a different colour), even though $H_R(I,M)$ does not vanish. We do not know any example of a ring $R$ and an $R$-module $M$ for which the conclusion of  Theorem \ref{corprinhom} fails.   
\begin{question}\label{que:inhomogeneous} Does there exist a ring $R$, an $R$-module $M$, a $k\times l$ matrix $\mathbf{A}$ with entries in $R$, and a nonzero $\mathbf{b} \in M^k$ such that the equation $\mathbf{Am}=\mathbf{b}$ is partition regular over $M$ even though it does not have any constant solutions in $M$? Can one choose $M=R$?
\end{question}

\section*{Note added in proof} We note that a complete answer to Question \ref{que:inhomogeneous} has now been obtained by Leader and Russell, who showed that for any $R$-module $M$ and any matrix $\mathbf{A}$ the equation $\mathbf{Am}=\mathbf{b}$ is partition regular over $M$ if and only if it has a constant solution \cite{LR-2020}.

\providecommand{\bysame}{\leavevmode\hbox to3em{\hrulefill}\thinspace}
\providecommand{\MR}{\relax\ifhmode\unskip\space\fi MR }
\providecommand{\MRhref}[2]{%
  \href{http://www.ams.org/mathscinet-getitem?mr=#1}{#2}
}
\providecommand{\href}[2]{#2}

\end{document}